\documentclass[11pt,leqno,twoside]{amsart}
\usepackage[margin=20mm]{geometry}
\usepackage{amsmath}
\usepackage{amsfonts}
\usepackage{enumitem}
\usepackage{amsthm}
\usepackage{bm}
\usepackage{mathtools}
\makeatletter \@addtoreset{equation}{section} \makeatother
\newcommand{\eref}[1]{(\ref{#1})}
\newcommand{\tref}[1]{Theorem \ref{#1}}
\newcommand{\pref}[1]{Proposition \ref{#1}}
\newcommand{\lref}[1]{Lemma \ref{#1}}
\newcommand{\rref}[1]{Remark \ref{#1}}
\newcommand{\sref}[1]{Section \ref{#1}}
\newcommand{\cref}[1]{Corollary \ref{#1}}

\theoremstyle{plain} \newtheorem{thm}{Theorem}[section] \newtheorem{lem}{Lemma}[section] \newtheorem{prop}{Proposition}[section] \newtheorem{cor}{Corollary}[section]
\theoremstyle{definition} \newtheorem{rem}{Remark}[section] 
\title[FEM for generalized Robin BCs in curved domains]{Finite element analysis of a generalized Robin boundary value problem in curved domains based on the extension method}
\author[Takahito Kashiwabara]{Takahito Kashiwabara}
\address{Graduate School of Mathematical Sciences, The University of Tokyo, 3-8-1 Komaba, Meguro, 153-8914 Tokyo, Japan}
\email{tkashiwa@ms.u-tokyo.ac.jp}
\date{\today}
\subjclass[2020]{Primary: 65N30}
\keywords{Finite element method; Generalized Robin boundary condition; Domain perturbation error; Extension method; Local coordinate representation; $H^1$-stable interpolation on boundary}
\thanks{This work was supported by a Grant-in-Aid for Early-Career Scientists (No.\ 20K14357) of the Japan Society for the Promotion of Science (JSPS)}

\begin{document}
\begin{abstract}
	A theoretical analysis of the finite element method for a generalized Robin boundary value problem, which involves a second-order differential operator on the boundary, is presented.
	If $\Omega$ is a general smooth domain with a curved boundary, we need to introduce an approximate domain $\Omega_h$ and to address issues owing to the domain perturbation $\Omega \neq \Omega_h$.
	In contrast to the transformation approach used in existing studies, we employ the extension approach, which is easier to handle in practical computation, in order to construct a numerical scheme.
	Assuming that approximate domains and function spaces are given by isoparametric finite elements of order $k$, we prove the optimal rate of convergence in the $H^1$- and $L^2$-norms.
	A numerical example is given for the piecewise linear case $k = 1$.
\end{abstract}
\maketitle

\section{Introduction}
The generalized Robin boundary value problem for the Poisson equation introduced in \cite{KCDQ2015} is described by
\begin{align}
	-\Delta u &= f \quad\text{in}\quad \Omega, \label{eq1: g-robin} \\
	\frac{\partial u}{\partial \bm n} + u - \Delta_{\Gamma} u &= \tau \quad\text{on}\quad \Gamma := \partial\Omega, \label{eq2: g-robin}
\end{align}
where $\Omega \subset \mathbb R^d$ is a smooth domain, $\bm n$ is the outer unit normal to $\Gamma$, and $\Delta_{\Gamma}$ stands for the Laplace--Beltrami operator defined on $\Gamma$.
Since elliptic equations in the bulk domain and on the surface are coupled through the normal derivative, it can be regarded as one of the typical models of coupled bulk-surface PDEs, cf.\ \cite{EllRan2013}.
It is also related to problems with dynamic boundary conditions \cite{KovLub2017} or to reduced-order models for fluid-structure interaction problems \cite{CDQ2014}.

Throughout this paper, we exploit the standard notation of the Sobolev spaces in the domain and on the boundary, that is, $W^{m,p}(\Omega)$ and $W^{m,p}(\Gamma)$ (written as $H^m(\Omega)$ and $H^m(\Gamma)$ if $p = 2$), together with the non-standard ones
\begin{equation*}
	H^m(\Omega; \Gamma) := \{ v \in H^m(\Omega) \mid v|_{\Gamma} \in H^m(\Gamma) \}, \quad \|v\|_{H^m(\Omega; \Gamma)} := \|v\|_{H^m(\Omega)} + \|v\|_{H^m(\Gamma)}.
\end{equation*}
According to \cite[Section 3.1]{KCDQ2015}, the weak formulation for \eref{eq1: g-robin}--\eref{eq2: g-robin} consists in finding $u \in H^1(\Omega; \Gamma)$ such that 
\begin{equation} \label{eq: continuous problem}
	(\nabla u, \nabla v)_{\Omega} + (u, v)_{\Gamma} + (\nabla_{\Gamma}u, \nabla_{\Gamma}v)_{\Gamma} = (f, v)_{\Omega} + (\tau, v)_{\Gamma} \qquad \forall v \in H^1(\Omega; \Gamma),
\end{equation}
where $(\cdot, \cdot)_\Omega$ and $(\cdot, \cdot)_\Gamma$ denote the $L^2(\Omega)$- and $L^2(\Gamma)$-inner products respectively, and $\nabla_\Gamma$ stands for the surface gradient along $\Gamma$.
It is shown in \cite{KCDQ2015} that this problem admits the following regularity structure for some constant $C > 0$:
\begin{equation*}
	\|u\|_{H^m(\Omega; \Gamma)} \le C (\|f\|_{H^{m-2}(\Omega)} + \|\tau\|_{H^{m-2}(\Gamma)}) \qquad (m = 2, 3, \dots).
\end{equation*}
Moreover, the standard finite element analysis is shown to be applicable, provided that either $\Omega$ is a polyhedral domain and \eref{eq2: g-robin} is imposed on a whole edge or face in $\Gamma$, or $\Omega$ is smooth and can be exactly represented in the framework of the isogeometric analysis.

For a more general smooth domain, a feasible setting is to exploit the $\mathbb P_k$-isoparametric finite element method, in which $\Gamma = \partial\Omega$ is approximated by piecewise polynomial (of degree $k$) boundary $\Gamma_h = \partial\Omega_h$.
Because the approximate domain $\Omega_h$ does not agree with $\Omega$, its theoretical analysis requires estimation of errors owing to the discrepancy of the two domains, i.e., the domain perturbation.

Such an error analysis is presented by \cite{KovLub2017} in a time-dependent case for $k = 1$ and by \cite{Ede2021} for $k \ge 1$, based on the \emph{transformation method}.
The name comes from the fact that they introduce a bijection $L_h : \Omega_h \to \Omega$ and ``lift'' a function $v : \Omega_h \to \mathbb R$ to $v \circ L_h^{-1} : \Omega \to \mathbb R$ defined in $\Omega$, thus transforming all functions so that they are defined in the original domain $\Omega$.
In this setting, the finite element scheme reads, with a suitable choice of the finite element space $V_h \subset H^1(\Omega_h; \Gamma_h)$: find $u_h \in V_h$ such that
\begin{equation} \label{eq: scheme by transformation method}
	(\nabla u_h, \nabla v_h)_{\Omega_h} + (u_h, v_h)_{\Gamma_h} + (\nabla_{\Gamma_h}u_h, \nabla_{\Gamma_h}v_h)_{\Gamma_h} = (f^{-l}, v_h)_{\Omega_h} + (\tau^{-l}, v_h)_{\Gamma_h} \qquad \forall v_h \in V_h,
\end{equation}
where $f^{-l} = f \circ L_h$ and $\tau^{-l} = \tau \circ L_h$ mean the inverse lifts of $f$ and $\tau$ respectively.
Then the error between the approximate and exact solutions are defined as $u - u_h^l$ on $\Omega$ with $u_h^l := u_h \circ L_h^{-1}$.
It is theoretically proved by \cite{Len86} and \cite{Ber1989} that such a transformation $L_h$ indeed exists.
However, from the viewpoint of practical computation, it does not seem easy to construct $L_h$ for general domains in a concrete way.
Therefore, it is non-trivial to numerically compute $f^{-l}, \tau^{-l}$, and $u - u_h^l$.

There is a more classical and direct approach to treat the situation $\Omega \neq \Omega_h$, which we call the \emph{extension method} (see e.g.\ \cite[Section 4.5]{Cia78} and \cite{BaEl88}; a more recent result is found in \cite{ChiSai2023}).
Namely, we extend $f$ and $\tau$ to some $\tilde f$ and $\tilde\tau$ which are defined in $\mathbb R^d$, preserving their smoothness (this can be justified by the Sobolev extension theorem or the trace theorem).
Then the numerical scheme reads: find $u_h \in V_h$ such that
\begin{equation*}
	(\nabla u_h, \nabla v_h)_{\Omega_h} + (u_h, v_h)_{\Gamma_h} + (\nabla_{\Gamma_h}u_h, \nabla_{\Gamma_h}v_h)_{\Gamma_h} = (\tilde f, v_h)_{\Omega_h} + (\tilde\tau, v_h)_{\Gamma_h} \qquad \forall v_h \in V_h,
\end{equation*}
and the error is defined as $\tilde u - u_h$ in the approximate domain $\Omega_h$.
If $f$ and $\tau$ are given as entire functions, which is often the case in practical computation, then no special treatment for them is needed.
Moreover, when computing errors numerically for verification purposes, it is usual to calculate $\tilde u - u_h$ in the computational domain $\Omega_h$ rather than $u - u_h^l$ in $\Omega$ simply because the former is easier to deal with.

In view of these situations, we aim to justify the use of the extension method for problem \eref{eq1: g-robin}--\eref{eq2: g-robin} in the present paper.
Considering $\Omega_h$ which approximates $\Omega$ by the $\mathbb P_k$-isoparametric elements, we establish in \sref{sec: error estimate} the following error estimates as the main result:
\begin{equation*}
	\|\tilde u - u_h\|_{H^1(\Omega_h; \Gamma_h)} \le O(h^{k}), \qquad \|\tilde u - u_h\|_{L^2(\Omega_h; \Gamma_h)} \le O(h^{k+1}).
\end{equation*}
They do not follow from the results of \cite{KovLub2017} or \cite{Ede2021} directly since we need to estimate errors caused from a transformation that are absent in the transformation method.
In addition, there is a completely non-trivial point that is specific to the boundary condition \eref{eq2: g-robin}:
even if $u \in H^2(\Omega)$ with $u|_\Gamma \in H^2(\Gamma)$, we may have only $\tilde u|_{\Gamma_h} \in H^{3/2}(\Gamma_h)$, which could cause loss in the rate of convergence on the boundary.
To overcome this technical difficulty, a delicate analysis of interpolation errors on $\Gamma_h$, including the use of the Scott--Zhang interpolation operator on the boundary, is necessary as presented in \sref{sec: FEM}.

There is another delicate point when comparing a quantity defined in $\Gamma_h$ with that in $\Gamma$.
For simplicity in the explanation, let $\Gamma_h$ be given as a piecewise linear ($k = 1$) approximation to $\Gamma$.
If $d = 2$ and every node (vertex) of $\Gamma_h$ lies exactly on $\Gamma$, then the orthogonal projection $\bm p: \Gamma \to \Gamma_h$ is bijective and it is reasonable to set a local coordinate along each boundary element $S \in \mathcal S_h$ (see Subsection \ref{subsec: approximate domains} for the notation).
Namely, $S$ and $\bm p^{-1}(S)$ are represented as graphs $(y_1, 0)$ and $(y_1, \varphi(y_1))$ respectively with a local coordinate $(y_1, y_2)$.

However, if nodes do not belong to $\Gamma$, then $\bm p$ is no longer injective (see Figure \ref{fig1}).
Furthermore, for $d \ge 3$ the same situation necessarily occurs---no matter if boundary nodes are in $\Gamma$ or not---since $\partial S$ (its dimension is $\ge 1$) is not exactly contained in $\Gamma$.
Consequently, it is inconsistent in general to assume the following simultaneously:
\begin{enumerate}[label=(\roman{*})]
	\item each $S \in \mathcal S_h$ has one-to-one correspondence to some subset $\Gamma_S \subset \Gamma$;
	\item both $S$ and $\Gamma_S$ admit graph representations in some rotated cartesian coordinate, whose domains of definition are the same;
	\item $\Gamma = \bigcup_{S \in \mathcal S_h} \Gamma_S$ is a disjoint union, that is, $\{\Gamma_S\}_{S \in \mathcal S_h}$ forms an exact partition of $\Gamma$.
\end{enumerate}
We remark that this inconsistency is sometimes overlooked in literature considering $\Omega \neq \Omega_h$.

\begin{figure}[htbp] \label{fig1}
	\centering
	\includegraphics[width=7cm]{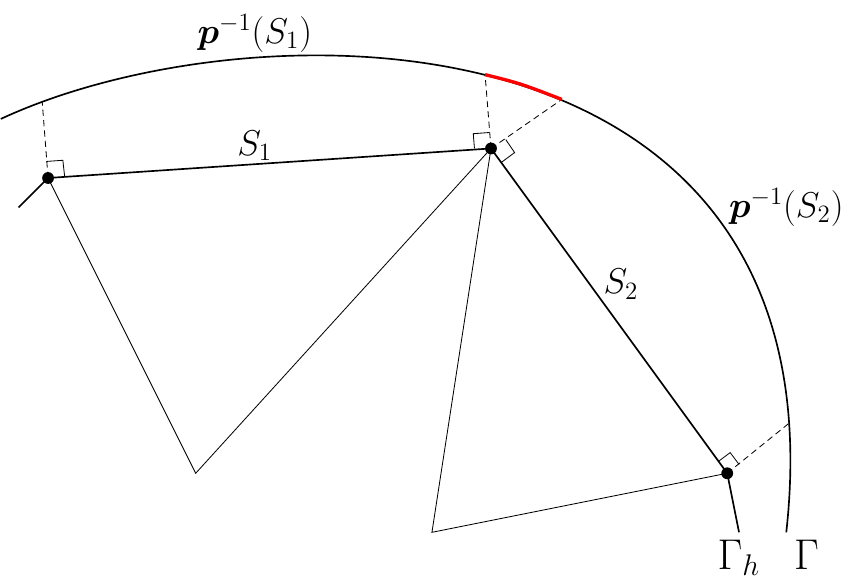}
	\hspace{1cm}
	\includegraphics[width=7cm]{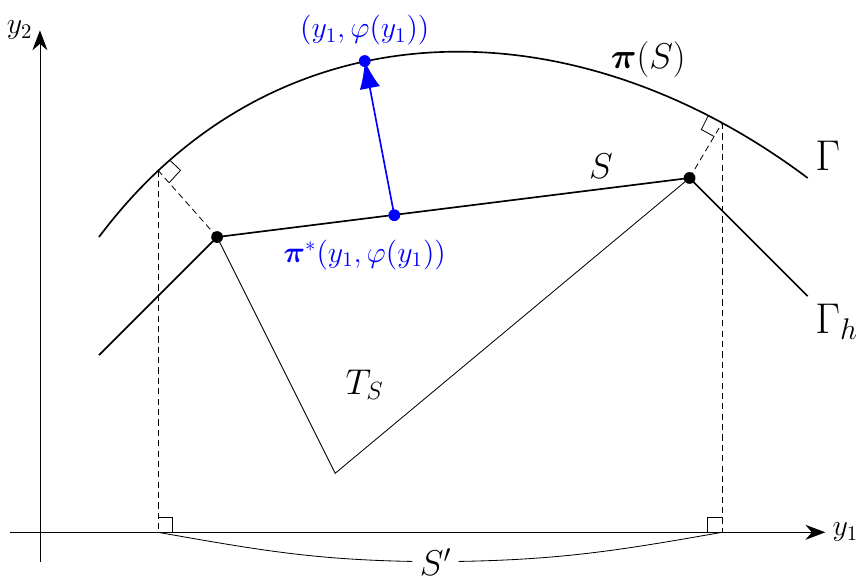}
	\caption{$\Gamma$ and $\Gamma_h$ for $d = 2$ and $k = 1$.
	Left: if $\partial S \not\subset \Gamma$, $\bm p$ is not injective (in the red part) and property (iii) fails to hold.
	Right: $\bm\pi(S)$ and $S$ are parametrized over the common domain $S'$. The representation of $\bm\pi(S)$ is a graph but that of $S$ is not.}
\end{figure}

To address the issue, we utilize the orthogonal projection $\bm\pi : \Gamma_h \to \Gamma$ (its precise definition is given in Subsection \ref{subsec: local coordinate}) instead of $\bm p$.
This map is bijective as long as $\Gamma_h$ is close enough to $\Gamma$, so that properties (i) and (iii) hold with $\Gamma_S = \bm\pi(S)$.
Then we set a local coordinate along $\bm\pi(S)$ and parametrize $S$ through $\bm\pi$ with the same domain as in Figure \ref{fig1}, avoiding the inconsistency above (we do not rely on a graph representation of $S$ in evaluating surface integrals etc.).

Finally, in Appendix \ref{sec: estimate in exact domain}, considering the so-called natural extension of $u_h$ to $\Omega$ denoted by $\bar u_h$, we also prove that $u - \bar u_h$ converges to $0$ at the optimal rate in $H^1(\Omega; \Gamma)$ and $L^2(\Omega; \Gamma)$ (actually there is some abuse of notation here; see \rref{rem: abuse of notation}).
This result may be regarded as an extension of \cite[Section 4.2.3]{Ric2017}, which discussed a Dirichlet problem for $d = 2$, to a more general setting.
Whereas it is of interest mainly from the mathematical point of view, it justifies calculating errors in approximate domains $\Omega_h, \Gamma_h$ based on extensions to estimate the rate of convergence in the original domains $\Omega, \Gamma$.

\section{Approximation and perturbation of domains}
\subsection{Assumptions on $\Omega$}
Let $\Omega \subset \mathbb R^d \, (d \ge 2)$ be a bounded domain of $C^{k+1,1}$-class ($k \ge 1$), with $\Gamma := \partial\Omega$.
Then there exist a system of local coordinates $\{(U_r, \bm y_r, \varphi_r)\}_{r=1}^M$ such that $\{U_r\}_{r=1}^M$ forms an open covering of $\Gamma$, 
$\bm y_r = {}^t(y_{r1}, \dots, y_{rd-1}, y_{rd}) = {}^t(\bm y_r', y_{rd})$ is a rotated coordinate of $\bm x$, and $\varphi_r: \Delta_r \to \mathbb R$ gives a graph representation $\bm \Phi_r(\bm y_r') := {}^t(\bm y_r', \varphi_r(\bm y_r'))$ of $\Gamma \cap U_r$, where $\Delta_r$ is an open cube in $\mathbb R^{N-1}$. 
Because $C^{k, 1}(\Delta_r) = W^{k+1, \infty}(\Delta_r)$, we may assume that
\begin{equation*}
	\|(\nabla')^{m} \varphi_r\|_{L^\infty(\Delta')} \le C \quad (m = 0, \dots, k+1, \; r = 1, \dots, M)
\end{equation*}
for some constant $C > 0$, where $\nabla'$ means the gradient with respect to $\bm y_r'$.

We also introduce a notion of tubular neighborhoods $\Gamma(\delta) := \{x\in\mathbb R^N \,:\, \operatorname{dist}(x, \Gamma) \le \delta\}$.
It is known that (see \cite[Section 14.6]{GiTr98}) there exists $\delta_0>0$, which depends on the $C^{1,1}$-regularity of $\Omega$, such that each $\bm x \in \Gamma(\delta_0)$ admits a unique representation
\begin{equation*}
	\bm x = \bar{\bm x} + t \bm n(\bar{\bm x}), \qquad \bar{\bm x} \in \Gamma, \, t \in [-\delta_0, \delta_0].
\end{equation*}
We denote the maps $\Gamma(\delta_0)\to \Gamma$; $\bm x\mapsto\bar{\bm x}$ and $\Gamma(\delta_0)\to \mathbb R$; $\bm x \mapsto t$ by $\bm\pi(\bm x)$ and $d(\bm x)$, respectively
(actually, $\bm\pi$ is an orthogonal projection to $\Gamma$ and $d$ agrees with the signed-distance function).
The regularity of $\Omega$ is transferred to that of $\bm\pi$, $d$, and $\bm n$ (cf.\ \cite[Section 7.8]{DeZo11}).
In particular, $\bm n \in \bm C^{k, 1}(\Gamma)$.

\subsection{Assumptions on approximate domains} \label{subsec: approximate domains}
We make the following assumptions (H1)--(H8) on finite element partitions and approximate domains.
First we introduce a regular family of triangulations $\{\tilde{\mathcal T}_h\}_{h \downarrow 0}$ of \emph{straight $d$-simplices} and define the set of nodes corresponding to the standard $\mathbb P_k$-finite element.
\begin{enumerate}[label=(H\arabic*)]
	\item Every $T \in \mathcal{\tilde T}_h$ is affine-equivalent to the standard closed simplex $\hat T$ of $\mathbb R^d$, 
	via the isomorphism $\tilde{\bm F}_{T}(\hat{\bm x}) = B_{T} \hat{\bm x} + \bm b_{T}$.
	The set $\tilde{\mathcal T}_h$ is mutually disjoint, that is, the intersection of every two different elements is either empty or agrees with their common face of dimension $\le d - 1$.
	
	\item $\{\tilde{\mathcal T}_h\}_{h \downarrow 0}$ is regular in the sense that
	\begin{equation*}
		h_T \le C \rho_T \quad (\forall h > 0, \, \forall T \in \mathcal T_h),
	\end{equation*}
	where $h_T$ and $\rho_T$ stand for the diameter of the smallest ball containing $T$ and that of the largest ball contained $T$, respectively.
	
	\item We let $\hat\Sigma_k = \{\hat{\bm a}_i\}_{i=1}^{N_k}$ denote the nodes in $\hat T$ of the continuous $\mathbb P_k$-finite element (see e.g.\ \cite[Section 2.2]{Cia78}).
	The nodal basis functions $\hat\phi_i \in \mathbb P_k(\hat T)$, also known as the shape functions, are then defined by $\hat\phi_i(\hat{\bm a}_j) = \delta_{ij}$ (the Kronecker delta) for $i, j = 1, \dots, N_k$.
\end{enumerate}

\begin{rem}
	If $\hat T$ is chosen as the standard $d$-simplex, i.e., $\hat T = \{(\hat x_1, \dots, \hat x_d) \in \mathbb R^d \mid x_1 \ge 0, \dots, x_d \ge 0, \hat x_1 + \cdots + \hat x_d \le 1\}$, then the standard position of the nodes for the $\mathbb P_k$-finite element is specified as $\hat\Sigma_k = \{ (\hat i_1/k, \dots, \hat i_d/k) \in \hat T \mid \hat i_1, \dots, \hat i_d \in \mathbb N_{\ge 0} \}$.
\end{rem}

We now introduce a partition into $\mathbb P_k$-isoparametric finite elements, denoted by $\mathcal T_h$, from $\tilde{\mathcal T}_h$, which results in approximate domains $\Omega_h$.
We assume that $\Omega_h$ is a perturbation of a polyhedral domain.
\begin{enumerate}[label=(H\arabic*)]
	\setcounter{enumi}{3}
	\item \label{H4}
	For $\tilde T \in \tilde{\mathcal T}_h$ we define a parametric map $\bm F \in [\mathbb P_k(\hat T)]^d$ by
	\begin{equation*}
		\bm F(\hat{\bm x}) = \sum_{i=1}^{N_k} \bm a_i \hat\phi_i(\hat{\bm x}),
	\end{equation*}
	where the ``mapped nodes'' $\bm a_i \in \mathbb R^d \, (i = 1, \dots, N_k)$ satisfy
	\begin{equation*}
		|\bm a_i - \bm F_{\tilde T}(\hat{\bm a}_i)| \le Ch_{\tilde T}^2.
	\end{equation*}
	If $h_{\tilde T}$ is small such $\bm F$ becomes diffeomorphic on $\hat T$ (see \cite[Theorem 3]{CiaRav1972}), and we set $T := \bm F(\hat T)$.
	For convenience in the notation, henceforth we write $\bm F$ as $\bm F_T$, $\tilde{\bm F}_{\tilde T}$ as $\tilde{\bm F}_T$, and $h_{\tilde T}$ as $h_T$.
	
	\item The partition $\mathcal T_h$ is defined as the set of $T$ constructed above.
	We define $\Omega_h$ to be the interior of the union of $\mathcal T_h$; in particular, $\overline\Omega_h = \bigcup_{T \in \mathcal T_h} T$.
	
	\item \label{H6}
	$\{\mathcal T_h\}_{h \downarrow 0}$ is regular of order $k$ in the sense of \cite[Definition 3.2]{Ber1989}, that is,
	\begin{equation*}
		\|\nabla_{\hat{\bm x}}^m \bm F_T\|_{L^\infty(\hat T)} \le C\|B_T\|_{\mathcal L(\mathbb R^d, \mathbb R^d)}^m \le Ch_T^m
			\qquad (T \in \mathcal T_h, \quad m = 2, \dots, k+1),
	\end{equation*}
	where $C$ is independent of $h$ (if $m = k + 1$ the left-hand side is obviously $0$).
\end{enumerate}
\begin{rem} \label{rem: after regularity of order k}
	(i) Throughout this paper, we assume without special emphasis that $h$ is sufficiently small; especially that $h \le 1$.

	(ii) \ref{H6} automatically holds if $\bm F_T$ is an $O(h^k)$-perturbation of $\tilde{\bm F}_T$ (see \cite[p.\ 239]{CiaRav1972}).
	It is a reasonable assumption for $k = 2$, but is not compatible with \ref{H8} below for $k \ge 3$, which is why we presume \ref{H6} independently.
	
	(iii) \cite{Len86} presented a procedure to construct $\mathcal T_h$ satisfying \ref{H4}--\ref{H6} for general $d$ and $k$, which is done inductively on $k$.
	In order to get, e.g., cubic isoparametric partitions with regularity of order 3, one needs to know a quadratic partition of order 2 in advance.
	Then, a kind of perturbation is added to the quadratic map to satisfy the condition of order 3 (see \cite[eq.\ (22)]{Len86}).
	
	(iv) As a result of \ref{H4}--\ref{H6}, for $T \in \mathcal T_h$ we have (see \cite[Theorems 3 and 4]{CiaRav1972} and \cite[Theorem 1]{Len86}):
	\begin{gather*}
		\|\nabla_{\hat{\bm x}} \bm F_T\|_{L^\infty(\hat T)} \le C\|B_T\|_{\mathcal L(\mathbb R^d, \mathbb R^d)} \le Ch_T, \\
		C_1 h_T^d \le |\operatorname{det} (\nabla_{\hat{\bm x}} \bm F_T)| \le C_2 h_T^d, \\
		\|\nabla_{\bm x}^m \bm F_T^{-1}\|_{L^\infty(T)} \le C \|B_T^{-1}\|_{\mathcal L(\mathbb R^d, \mathbb R^d)}^m \le C h_T^{-m} \quad (m = 1, \dots, k+1).
	\end{gather*}
\end{rem}

We next introduce descriptions on boundary meshes.
Setting $\Gamma_h := \partial\Omega_h$, we define the boundary mesh $\mathcal S_h$ inherited from $\mathcal T_h$ by
\begin{equation*}
	\mathcal S_h = \{ S \subset \Gamma_h \mid \text{$S = \bm F_T(\hat S)$ for some $T \in \mathcal T_h$, where $\hat S \subset \partial\hat T$ is a $(d-1)$-face of $\hat T$} \}.
\end{equation*}
Then we have $\Gamma_h = \bigcup_{S \in \mathcal T_h} S$ (disjoint union).
Each boundary element $S \in \mathcal S_h$ admits a unique $T \in \mathcal T_h$ such that $S \subset \partial T$, which is denoted by $T_S$.
We let $\bm b_r : U_r \to \mathbb R^{d-1}; {}^t(\bm y_r', y_{rd}) \mapsto \bm y_r'$ denote the projection to the base set.
Let us now assume that $\Omega$ is approximated by $\Omega_h$ in the following sense.
\begin{enumerate}[label=(H\arabic*)]
	\setcounter{enumi}{6}
	\item \label{asmp: omit r}
	$\Gamma_h$ is covered by $\{U_r\}_{r=1}^M$, and each portion $\Gamma_h \cap U_r$ is represented as a graph $(\bm y_r', \varphi_{rh}(\bm y_r'))$, where $\varphi_{rh}$ is a continuous function defined in $\overline{\Delta_r}$.
	Moreover, each $S \in \mathcal S_h$ is contained in some $U_r$.
	We fix such $r$ and agree to omit the subscript $r$ for simplicity when there is no fear of confusion.
	
	\item \label{H8} The restriction of $\varphi_{rh}$ to $\bm b_r(S)$ for each $S \in \mathcal S_h$ is a polynomial function of degree $\le k$.
	Moreover, $\varphi_{rh}$ approximates $\varphi_r$ as accurately as a general $\mathbb P_k$-interpolation does; namely, we assume that
	\begin{align}
		\|\varphi_r - \varphi_{rh}\|_{L^\infty(\bm b_r(S))} &\le Ch_{S}^{k+1} =: \delta_S, \label{eq: phi - phih} \\
		\|(\nabla')^m (\varphi_r - \varphi_{rh})\|_{L^\infty(\bm b_r(S))} &\le Ch_{S}^{k+1-m} \qquad (m = 1, \dots, k + 1), \label{eq: derivatives of phi - phih}
	\end{align}
	where the boundary mesh size is defined as $h_S := h_{T_S}$.
\end{enumerate}

These assumptions essentially imply that the local coordinate system for $\Omega$ is compatible with $\{\Omega_h\}_{h \downarrow 0}$ and that $\Gamma_h$ is a piecewise $\mathbb P_k$ interpolation of $\Gamma$.
Setting $\delta := \max_{S \in \mathcal S_h} \delta_S$, we have $\operatorname{dist}(\Gamma, \Gamma_h) \le \delta < \delta_0$ if $h$ is sufficiently small, so that $\bm\pi$ is well-defined on $\Gamma_h$.

\subsection{Local coordinates for $\Gamma$ and $\Gamma_h$} \label{subsec: local coordinate}
In \cite[Proposition 8.1]{KOZ16}, we proved that $\bm\pi|_{\Gamma_h}$ gives a homeomorphism (and element-wisely a diffeomorphism) between $\Gamma$ and $\Gamma_h$ provided $h$ is sufficiently small, taking advantage of the fact that $\Gamma_h$ can be regarded as a $\mathbb P_k$-interpolation of $\Gamma$ (there we assumed $k = 1$, but the method can be easily adapted to general $k \ge 1$).
If we write its inverse map $\bm\pi^* : \Gamma\to\Gamma_h$ as $\bm\pi^*(\bm x) = \bar{\bm x} + t^*(\bar{\bm x}) \bm n(\bar{\bm x})$, then $t^*$ satisfies (cf.\ \cite[Proposition 8.2]{KOZ16})
\begin{equation} \label{eq: global t*}
	\| t^*\|_{L^\infty(\Gamma)} \le \delta, \qquad \|\nabla_\Gamma^m t^*\|_{L^\infty(\Gamma)} \le Ch^{k+1-m} \quad (m = 1, \dots, k + 1),
\end{equation}
corresponding to \eref{eq: phi - phih} and \eref{eq: derivatives of phi - phih}.
Here, $\nabla_\Gamma$ means the surface gradient along $\Gamma$ and the constant depends only on the $C^{1,1}$-regularity of $\Omega$.
This in particular implies that $\Omega_h\triangle\Omega := (\Omega_h\setminus\Omega) \cup (\Omega\setminus\Omega_h)$ and $\Gamma_h \cup \Gamma$ are contained in $\Gamma(\delta)$.
We refer to $\Omega_h\triangle\Omega$, $\Gamma(\delta)$ and their subsets as \emph{boundary-skin layers} or more simply as \emph{boundary skins}.

For $S\in \mathcal S_h$, we may assume that $S \cup \bm\pi(S)$ is contained in some local coordinate neighborhood $U_r$.
As announced in \ref{asmp: omit r} above, we will omit the subscript $r$ in the subsequent argument.
We define
\begin{equation*}
	S' := \bm b(\bm\pi(S)) \quad (\text{note that it differs from $\bm b(S)$})
\end{equation*}
to be the common domain of parameterizations of $\bm\pi(S) \subset \Gamma$ and $S \subset \Gamma_h$.
In fact, $\bm\Phi : S' \to \bm\pi(S)$ and $\bm\Phi_{h} := \bm\pi^* \circ \bm\Phi : S' \to S$ constitute smooth (at least $C^{k,1}$) bijections.
We then obtain $\bm\pi^*(\bm\Phi(\bm z')) = \bm\Phi(\bm z') + t^*(\bm\Phi(\bm z')) \bm n(\bm\Phi(\bm z'))$ for $\bm z' \in S'$ and
\begin{equation*}
	\|t^* \circ \bm\Phi\|_{L^\infty(S')} \le \delta_S, \qquad \|(\nabla')^m (t^* \circ \bm\Phi)\|_{L^\infty(S')} \le Ch_{S}^{k+1-m} \quad (m = 1, \dots, k + 1),
\end{equation*}
which are localized versions of \eref{eq: global t*}.

Let us represent integrals associated with $S$ in terms of the local coordinates introduced above.
First, surface integrals along $\bm\pi(S)$ and $S$ are expressed as
\begin{align*}
	\int_{\bm\pi(S)} v \, d\gamma = \int_{S'} v(\bm\Phi(\bm y')) \sqrt{\operatorname{det} G(\bm y')} \, d\bm y', \qquad \int_{S} v \, d\gamma_h = \int_{S'} v(\bm\Phi_h(\bm y')) \sqrt{\operatorname{det} G_h(\bm y')} \, d\bm y',
\end{align*}
where $G$ and $G_h$ denote the Riemannian metric tensors obtained from the parameterizations $\bm\Phi$ and $\bm\Phi_h$, respectively.
Namely, for tangent vectors $\bm g_\alpha := \frac{\partial \bm\Phi}{\partial z_\alpha}$ and $\bm g_{h, \alpha} := \frac{\partial \bm\Phi_h}{\partial z_\alpha}$ $(\alpha = 1, \dots, d-1)$, the components of and $G$ and $G_h$, which are $(d-1)\times(d-1)$ matrices, are given by
\begin{equation*}
	G_{\alpha \beta} = \bm g_\alpha \cdot \bm g_\beta, \qquad G_{h, \alpha \beta} = \bm g_{h, \alpha} \cdot \bm g_{h, \beta}.
\end{equation*}
The contravariant components of the metric tensors and the contravariant vectors on $\Gamma$ are defined as
\begin{equation*}
	G^{\alpha \beta} = (G^{-1})_{\alpha \beta}, \qquad \bm g^{\alpha} = \sum_{\beta=1}^{d-1} G^{\alpha \beta} \bm g_\beta,
\end{equation*}
together with their counterparts $G_h^{\alpha, \beta}$ and $\bm g_h^\alpha$ on $\Gamma_h$.
Then the surface gradients along $\Gamma$ and $\Gamma_h$ can be represented in the local coordinate as (see \cite[Lemma 2.1]{KCDQ2015})
\begin{equation} \label{eq: local representation of surface gradient}
	\nabla_{\Gamma} = \sum_{\alpha=1}^{d-1} \bm g^\alpha \frac{\partial}{\partial z_\alpha}, \qquad \nabla_{\Gamma_h} = \sum_{\alpha=1}^{d-1} \bm g_h^\alpha \frac{\partial}{\partial z_\alpha}.
\end{equation}

In the same way as we did in \cite[Theorem 8.1]{KOZ16}, we can show $\|\bm g_\alpha - \bm g_{h,\alpha}\|_{L^\infty(S')} \le Ch_S^k$ and $\|G_{\alpha\beta} - G_{h, \alpha \beta}\|_{L^\infty(S')} \le C\delta_S$.
We then have $\|G^{\alpha\beta} - G_h^{\alpha \beta}\|_{L^\infty(S')} \le C\delta_S$, because
\begin{equation*}
	G_h^{-1} - G^{-1} = G^{-1}\underbrace{(G_h - G)}_{=O(\delta_S)}G_h^{-1}.
\end{equation*}
Note that the stability of $G_h^{-1}$ follows from the representation $G_h = G (I + G^{-1} X)$, with $X = G_h - G$ denoting a perturbation, together with a Neumann series argument.
As a result, one also gets an error estimate for contravariant vectors, i.e., $\|\bm g^\alpha - \bm g_h^\alpha\|_{L^\infty(S')} \le Ch_S^k$.

Derivative estimates for metric tensors and vectors can be derived as well for $m = 1, \dots, k$:
\begin{equation} \label{eq: derivative of G - Gh}
\begin{aligned}
	\|G_{\alpha\beta} - G_{h, \alpha \beta}\|_{W^{m, \infty}(S')} &\le Ch_S^{k-m}, \qquad \|G^{\alpha\beta} - G_h^{\alpha \beta}\|_{W^{m, \infty}(S')} \le Ch_S^{k-m}, \qquad \\
	\|\bm g_\alpha - \bm g_{h,\alpha}\|_{W^{m, \infty}(S')} &\le Ch_S^{k-m}, \qquad \|\bm g^\alpha - \bm g^{h,\alpha}\|_{W^{m, \infty}(S')} \le Ch_S^{k-m}.
\end{aligned}
\end{equation}

Next, let $\bm\pi(S, \delta) := \{\bar{\bm x} + t \bm n(\bar{\bm x}) \mid \bar{\bm x} \in \bm\pi(S), \; -\delta\le t\le \delta \}$ be a tubular neighborhood with the base $\bm\pi(S)$, and consider volume integrals over $\bm\pi(S, \delta)$.
To this end we introduce a one-to-one transformation $\bm\Psi: S'\times [-\delta, \delta] \to \bm\pi(S, \delta)$ by
\begin{equation*}
	\bm x = \bm\Psi(\bm z', t) := \bm\Phi(\bm z') + t \bm n(\bm\Phi(\bm z')) \Longleftrightarrow \bm z' = \bm b(\bm\pi(\bm x)), \; t = d(\bm x),
\end{equation*}
where we recall that $\bm b : \mathbb R^d \to \mathbb R^{d-1}$ is the projection.
Then, by change of variables, we obtain
\begin{equation*}
	\int_{\bm \pi(S,\delta)} v(\bm x) \, d\bm x = \int_{S' \times [-\delta, \delta]} v(\bm\Psi(\bm z', t)) |\operatorname{det} J(z', t)| \, d\bm z' dt,
\end{equation*}
where $J := \nabla_{(\bm z', t)} \bm\Psi$ denotes the Jacobi matrix of $\bm\Psi$.
In the formulas above, $\operatorname{det}G$, $\operatorname{det}G_h$, and $\operatorname{det}J$ can be bounded, from above and below, by positive constants depending on the $C^{1,1}$-regularity of $\Omega$, provided $h$ is sufficiently small.
In particular, we obtain the following equivalence estimates:
\begin{align}
	&C_1 \int_{\bm\pi(S)} |v| \, d\gamma \le \int_{S'} |v \circ \bm\Phi| \, d\bm z' \le C_2 \int_{\bm\pi(S)} |v| \, d\gamma, \label{eq: local surface integral on Gamma} \\
	&C_1 \int_{S} |v| \, d\gamma_h \le \int_{S'} |v \circ \bm\Phi_h| \, d\bm z' \le C_2 \int_{S} |v| \, d\gamma_h, \label{eq: local surface integral on Gammah} \\
	&C_1 \int_{\bm\pi(S, \delta)} |v| \, d\bm x \le \int_{S' \times [-\delta, \delta]} |v \circ \bm\Psi| \,d\bm z' dt  \le C_2 \int_{\bm\pi(S, \delta)} |v| \, d\bm x. \label{eq: local tubular neighborhood equivalence}
\end{align}
We remark that the width $\delta$ in \eref{eq: local tubular neighborhood equivalence} may be replaced with arbitrary $\delta' \in [\delta_S, \delta]$.

We also state an equivalence relation between $W^{m,p}(\Gamma)$ and $W^{m,p}(\Gamma_h)$ when the transformation $\bm\pi$ is involved.
\begin{lem} \label{lem: equivalence of Sobolev spaces on Gamma and Gammah}
	Let $m =0, \dots, k + 1$ and $1\le p\le \infty$.
	For $S \in \mathcal S_h$ and $v \in W^{m, p}(\bm\pi(S))$, we have
	\begin{align}
		C_1 \|v\|_{L^p(\bm\pi(S))} &\le \|v \circ \bm\pi\|_{L^p(S)} \le C_2 \|v\|_{L^p(\bm\pi(S))}, \label{eq1: equivalence of surface integrals} \\
		C_1 \|\nabla_{\Gamma} v\|_{L^p(\bm\pi(S))} &\le \|\nabla_{\Gamma_h} (v \circ \bm\pi)\|_{L^p(S)} \le C_2 \|\nabla_{\Gamma} v\|_{L^p(\bm\pi(S))}, \label{eq2: equivalence of surface integrals} \\
		C_1 \|v\|_{W^{m, p}(\bm\pi(S))} &\le \|v \circ \bm\pi\|_{W^{m, p}(S)} \le C_2 \|v\|_{W^{m, p}(\bm\pi(S))} \quad (m \ge 2). \label{eq3: equivalence of surface integrals}
	\end{align}
\end{lem}
\begin{proof}
	Estimate \eref{eq1: equivalence of surface integrals} follows from \eref{eq: local surface integral on Gamma} and \eref{eq: local surface integral on Gammah} combined with $\bm\Phi_h = \bm\pi^* \circ \bm\Phi \Longleftrightarrow \bm\pi \circ \bm\Phi_h = \bm\Phi$.
	To obtain derivative estimates \eref{eq2: equivalence of surface integrals} and \eref{eq3: equivalence of surface integrals}, it suffices to notice that we can invert \eref{eq: local representation of surface gradient} as
	\begin{equation*}
		\frac{\partial}{\partial z_\alpha} = \sum_{\beta = 1}^{d-1} G_{\alpha\beta} (\bm g^\beta \cdot \nabla_{\bm\pi(S)}), \qquad 
		\frac{\partial}{\partial z_\alpha} = \sum_{\beta = 1}^{d-1} G_{h, \alpha\beta} (\bm g_h^\beta \cdot \nabla_S),
	\end{equation*}
	and that the derivatives of $G_{h, \alpha \beta}, G_h^{\alpha \beta}, \bm g_{h, \alpha}, \bm g_h^\alpha$ up to the $k$-th order are bounded independently of $h$ in $L^\infty(S')$, due to \eref{eq: derivative of G - Gh} and $h_S \le 1$.
\end{proof}

\subsection{Estimates for domain perturbation errors}
We recall the following boundary-skin estimates for $S \in \mathcal S_h$, $1\le p\le \infty$, and $v \in W^{1,p}(\Omega \cup \Gamma(\delta))$ (note that $\Omega \cup \Gamma(\delta) \supset \Omega \cup \Omega_h$):
\begin{align}
	&\left| \int_{\bm\pi(S)} v\,d\gamma - \int_{S} v\circ\bm\pi\,d\gamma_h \right| \le C\delta_S \|v\|_{L^1(\bm\pi(S))}, \label{eq1: boundary-skin estimates} \\
	&\|v\|_{L^p(\bm\pi(S, \delta'))} \le C(\delta'^{1/p} \|v\|_{L^p(\bm\pi(S))} + \delta' \|\nabla v\|_{L^p(\bm\pi(S, \delta'))}) \quad (\delta' \in [\delta_S, \delta]), \label{eq2: boundary-skin estimates} \\
	&\|v - v\circ\bm\pi\|_{L^p(S)} \le C\delta_S^{1-1/p} \|\nabla v\|_{L^p(\bm\pi(S, \delta_S))}. \label{eq3: boundary-skin estimates}
\end{align}
The proofs are given in \cite[Theorems 8.1--8.3]{KOZ16} for the case $k = 1$, which can be extended to $k \ge 2$ without essential difficulty.
As a version of \eref{eq1: boundary-skin estimates}--\eref{eq3: boundary-skin estimates}, we also have
\begin{align}
	\left| \int_{\bm\pi(S)} v\circ\bm\pi^* \,d\gamma - \int_{S} v \,d\gamma_h \right| &\le C\delta_S \|v\|_{L^1(S)}, \notag \\
	\|v\|_{L^p(\bm\pi(S, \delta))} &\le C(\delta_S^{1/p} \|v\|_{L^p(S)} + \delta_S \|\nabla v\|_{L^p(\bm\pi(S, \delta))}), \label{eq2': boundary-skin estimates} \\
	\|v\circ\bm\pi^* - v\|_{L^p(\bm\pi(S))} &\le C\delta_S^{1-1/p} \|\nabla v\|_{L^p(\bm\pi(S, \delta))}. \notag
\end{align}
Adding up these for $S \in \mathcal S_h$ yields corresponding global estimates on $\Gamma$ or $\Gamma(\delta)$.
The following estimate limited to $\Omega_h \setminus \Omega$, rather than the whole boundary skin $\Gamma(\delta)$, also holds:
\begin{equation} \label{eq: RHS with Omegah minus Omega}
	\|v\|_{L^p(\Omega_h \setminus \Omega)} \le C(\delta^{1/p} \|v\|_{L^p(\Gamma_h)} + \delta \|\nabla v\|_{L^p(\Omega_h \setminus \Omega)}),
\end{equation}
which is proved in \cite[Lemma A.1]{KasKem2020a}.
Finally, denoting by $\bm n_h$ the outward unit normal to $\Gamma_h$, we notice that its error compared with $\bm n$ is estimated as (see \cite[Lemma 9.1]{KOZ16})
\begin{equation} \label{eq: n - nh}
	\|\bm n\circ\bm\pi - \bm n_h\|_{L^\infty(S)} \le Ch_S^k.
\end{equation}

We now state a version of \eref{eq3: boundary-skin estimates} which involves the surface gradient.
The proof will be given in Appendix \ref{apx: nablaGammah(v - vpi)}.
\begin{lem} \label{lem: nablaGammah(v - vpi)}
	Let $S \in \mathcal S_h$ and $v \in W^{2,p}(\Omega \cup \Gamma(\delta))$ for $1\le p\le \infty$.
	Then we have
	\begin{align}
		\|\nabla_{\Gamma_h} (v - v\circ\bm\pi) \|_{L^p(S)} &\le Ch_S^k \|\nabla v\|_{L^p(S)} + C \delta_S^{1-1/p} \|\nabla^2 v\|_{L^p(\bm\pi(S, \delta_S))}, \label{eq1: conclusion of lemma which bounds nabla(v - v circ pi)} \\
		\|\nabla_{\Gamma_h} (v - v\circ\bm\pi) \|_{L^p(S)} &\le Ch_S^k \|\nabla v\|_{L^p(\bm\pi(S))} + C \delta_S^{1-1/p} \|\nabla^2 v\|_{L^p(\bm\pi(S, \delta_S))}. \label{eq2: conclusion of lemma which bounds nabla(v - v circ pi)}
	\end{align}
\end{lem}

\begin{cor} \label{cor: u - u circ pi on Gammah}
	Let $m = 0, 1$ and assume that $v \in H^{2}(\Omega \cup \Gamma(\delta))$ if $k = 1$ and that $v \in H^{3}(\Omega \cup \Gamma(\delta))$ if $k \ge 2$.
	Then we have
	\begin{equation*}
		\| v - v\circ\bm\pi\|_{H^m(\Gamma_h)} \le Ch^{k+1-m} \|v\|_{H^{\min\{k+1, 3\}}(\Omega \cup \Gamma(\delta))}.
	\end{equation*}
\end{cor}
\begin{proof}
	By virtue of \eref{eq2: boundary-skin estimates} and \eref{eq3: boundary-skin estimates} (more precisely, their global versions) we have
	\begin{align*}
		\| v - v\circ\bm\pi\|_{L^2(\Gamma_h)} &\le C \delta^{1/2} \|\nabla v\|_{L^2(\Gamma(\delta))}
			\le C \delta^{1/2} (\delta^{1/2} \|\nabla v\|_{L^2(\Gamma)} + \delta \|\nabla^2  v\|_{L^2(\Gamma(\delta))})
			\le C \delta \|v\|_{H^2(\Omega \cup \Gamma(\delta))}.
	\end{align*}
	Similarly, we see from \eref{eq2: conclusion of lemma which bounds nabla(v - v circ pi)} that
	\begin{align*}
		\|\nabla_{\Gamma_h} ( v - v\circ\bm\pi)\|_{L^2(\Gamma_h)} &\le C h^k (\|\nabla v\|_{L^2(\Gamma)} + \delta^{1/2} \|\nabla^2 v\|_{L^2(\Gamma(\delta))} ) \\
			&\le \begin{cases}
				Ch \|v\|_{H^2(\Omega)} + Ch \|\nabla^2  v\|_{L^2(\Omega \cup \Gamma(\delta))} & (k = 1) \\
				Ch^k \|v\|_{H^2(\Omega)} + C\delta^{1/2} (\delta^{1/2}\|\nabla^2 v\|_{L^2(\Gamma)} + \delta \|\nabla^3 v\|_{L^2(\Gamma(\delta))}) & (k \ge 2)
			\end{cases} \\
			&\le \begin{cases}
				Ch \|v\|_{H^2(\Omega \cup \Gamma(\delta))} & (k = 1) \\
				Ch^k \|v\|_{H^3(\Omega \cup \Gamma(\delta))} & (k \ge 2),
			\end{cases}
	\end{align*}
	where we have used $\delta = Ch^{k+1}$ and $h \le 1$.
\end{proof}
Below several lemmas are introduced to address errors related with the $L^2$-inner product on surfaces.
\begin{lem}
	For $u, v \in H^2(\Omega \cup \Gamma(\delta))$ we have
	\begin{equation*}
		|(u, v)_{\Gamma_h} - (u, v)_\Gamma| \le C \delta \|u\|_{H^2(\Omega \cup \Gamma(\delta))} \|v\|_{H^2(\Omega \cup \Gamma(\delta))}.
	\end{equation*}
\end{lem}
\begin{proof}
	Observe that
	\begin{equation*}
		(u, v)_{\Gamma_h} - (u, v)_\Gamma = (u - u\circ\bm\pi, v)_{\Gamma_h} + \big[ (u\circ\bm\pi, v)_{\Gamma_h} - (u, v\circ\bm\pi^*)_\Gamma \big] + (u, v\circ\bm\pi^* - v)_\Gamma.
	\end{equation*}
	The first term in the right-hand side is bounded by $C  \delta \|\tilde u\|_{H^2(\Omega \cup \Gamma(\delta))} \|v\|_{L^2(\Gamma_h)}$ due to \cref{cor: u - u circ pi on Gammah}.
	The third term can be treated similarly.
	From \eref{eq1: boundary-skin estimates} and \eref{eq1: equivalence of surface integrals} the second term is bounded by
	\begin{equation*}
		C\delta \|u (v\circ\bm\pi^*)\|_{L^1(\Gamma)} \le C \delta \|u\|_{L^2(\Gamma)} \|v\circ\bm\pi^*\|_{L^2(\Gamma)} \le C \delta \|u\|_{L^2(\Gamma)} \|v\|_{L^2(\Gamma_h)}. 
	\end{equation*}
	Using trace inequalities on $\Gamma$ and $\Gamma_h$, we arrive at the desired estimate.
\end{proof}
\begin{lem} \label{lem: error from integration by parts on Gammah}
	For $u \in H^2(\Gamma)$ and $v \in H^1(\Gamma_h)$ we have
	\begin{equation*}
		\big| ( (\Delta_\Gamma u)\circ\bm\pi, v )_{\Gamma_h} + (\nabla_{\Gamma_h}(u\circ\bm\pi), \nabla_{\Gamma_h}v)_{\Gamma_h} \big|
			\le C\delta (\|u\|_{H^2(\Gamma)} \|v\|_{L^2(\Gamma_h)} + \|\nabla_\Gamma u\|_{L^2(\Gamma)} \|\nabla_{\Gamma_h} v\|_{L^2(\Gamma_h)}).
	\end{equation*}
\end{lem}
\begin{proof}
	Using an integration-by-parts formula on $\Gamma$, we decompose the left-hand side as
	\begin{align*}
		&( (\Delta_\Gamma u)\circ\bm\pi, v )_{\Gamma_h} + (\nabla_{\Gamma_h}(u\circ\bm\pi), \nabla_{\Gamma_h}v)_{\Gamma_h} \\
		= \; &\big[ ( (\Delta_\Gamma u)\circ\bm\pi, v )_{\Gamma_h} - (\Delta_\Gamma u, v\circ\bm\pi^*)_\Gamma \big]
			+ \big[ -(\nabla_\Gamma u, \nabla_\Gamma(v\circ\bm\pi^*))_\Gamma + (\nabla_{\Gamma_h}(u\circ\bm\pi), \nabla_{\Gamma_h}v)_{\Gamma_h} \big] \\
		=: \; &I_1 + I_2.
	\end{align*}
	By \eref{eq1: boundary-skin estimates} and \eref{eq1: equivalence of surface integrals}, $|I_1| \le C\delta \|(\Delta_\Gamma u)\circ\bm\pi\|_{L^2(\Gamma_h)} \|v\|_{L^2(\Gamma_h)} \le Ch^2 \|u\|_{H^2(\Gamma)} \|v\|_{L^2(\Gamma_h)}$.
	For $I_2$, we represent the surface integrals on $S$ and $\bm\pi(S)$ based on the local coordinate as follows:
	\begin{align*}
		\int_S \nabla_{\Gamma_h}(u\circ\bm\pi) \cdot \nabla_{\Gamma_h} v \, d\gamma_h &= \int_{S'} \sum_{\alpha,\beta} \partial_\alpha (u \circ \bm\Phi) \partial_\beta (v \circ \bm\Phi_h) \, G_h^{\alpha\beta} \sqrt{\operatorname{det}G_h} \, dz', \\
		\int_{\bm\pi(S)} \nabla_{\Gamma}u \cdot \nabla_{\Gamma} (v\circ\bm\pi^*) \, d\gamma &= \int_{S'} \sum_{\alpha,\beta} \partial_\alpha (u \circ \bm\Phi) \partial_\beta (v \circ \bm\Phi_h) \, G^{\alpha\beta} \sqrt{\operatorname{det}G} \, dz'.
	\end{align*}
	Since $\|G - G_h\|_{L^\infty(S')} \le C\delta_S$, their difference is estimated by
	\begin{equation*}
		C\delta_S \|\nabla_{\bm z'}(u\circ\bm\Phi)\|_{L^2(S')} \|\nabla_{\bm z'}(v\circ\bm\Phi_h)\|_{L^2(S')} \le C\delta_S \|\nabla_\Gamma u\|_{L^2(\bm\pi(S))} \|\nabla_{\Gamma_h} v\|_{L^2(S)}.
	\end{equation*}
	Adding this up for $S \in \mathcal S_h$ gives $|I_2| \le C\delta \|\nabla_\Gamma u\|_{L^2(\Gamma)} \|\nabla_{\Gamma_h} v\|_{L^2(\Gamma_h)}$, and this completes the proof.
\end{proof}

\begin{rem}
	(i) Since $\Gamma_h$ itself is not $C^{1,1}$-smooth globally, $(-\Delta_{\Gamma_h} u, v) = (\nabla_{\Gamma_h} u, \nabla_{\Gamma_h} v)$ does not hold in general (see \cite[Lemma 3.1]{KCDQ2015}).
	
	(ii) An argument similar to the proof above shows, for $u, v \in H^1(\Gamma)$,
	\begin{equation}
		\big| ( \nabla_{\Gamma_h}(u\circ\bm\pi), \nabla_{\Gamma_h}(v\circ\bm\pi) )_{\Gamma_h}  - (\nabla_\Gamma u, \nabla_\Gamma v)_\Gamma \big|
			\le C \delta \|\nabla_\Gamma u\|_{L^2(\Gamma)} \|\nabla_\Gamma v\|_{L^2(\Gamma)}. \label{eq: error between nabla Gammah and nabla Gamma}
	\end{equation}
\end{rem}

\begin{lem} \label{lem: inner product between nabla(u - u circ pi) and nabla v circ pi}
	Let $u \in H^2(\Omega \cup \Gamma(\delta))$ and $v \in H^2(\Gamma)$.  Then we have
	\begin{equation*}
		\big| (\nabla_{\Gamma_h} (u - u\circ\bm\pi), \nabla_{\Gamma_h} (v\circ\bm\pi))_{\Gamma_h} \big| \le C\delta \|u\|_{H^2(\Omega \cup \Gamma(\delta))}\|v\|_{H^2(\Gamma)}.
	\end{equation*}
\end{lem}
\begin{proof}
	By \eref{eq: error between nabla Gammah and nabla Gamma},
	\begin{equation*}
		\big| (\nabla_{\Gamma_h} (u - u\circ\bm\pi), \nabla_{\Gamma_h} (v\circ\bm\pi))_{\Gamma_h} - (\nabla_{\Gamma} (u\circ\bm\pi^* - u), \nabla_{\Gamma} v)_{\Gamma} \big| \le C\delta (\|u\|_{H^1(\Gamma_h)} + \|u\|_{H^1(\Gamma)}) \|v\|_{H^1(\Gamma)}.
	\end{equation*}
	Next we observe that
	\begin{align*}
		|(\nabla_{\Gamma} (u\circ\bm\pi^* - u), \nabla_{\Gamma} v)_{\Gamma}| &= |(u\circ\bm\pi^* - u, \Delta_\Gamma v)_{\Gamma}| \le \|u\circ\bm\pi^* - u\|_{L^2(\Gamma)} \|v\|_{H^2(\Gamma)} \\
			&\le C \|u - u\circ\bm\pi\|_{L^2(\Gamma_h)} \|v\|_{H^2(\Gamma)}.
	\end{align*}
	This combined with the boundary-skin estimate
	\begin{equation*}
		\|u - u\circ\bm\pi\|_{L^2(\Gamma_h)} \le C\delta^{1/2} \|\nabla u\|_{L^2(\Gamma(\delta))} \le C\delta^{1/2} (\delta^{1/2} \|\nabla u\|_{H^1(\Gamma)} + \delta \|\nabla^2 u\|_{L^2(\Gamma(\delta))}),
	\end{equation*}
	with the trace theorem in $\Omega$, and with $\delta \le 1$, yields the desired estimate.
\end{proof}

\section{Finite element approximation} \label{sec: FEM}
\subsection{Finite element spaces}
We introduce the global nodes of $\mathcal T_h$ by
\begin{equation*}
	\mathcal N_h = \{ \bm F_T(\hat{\bm a}_i) \in \overline\Omega_h \mid T \in \mathcal T_h, \; i = 1, \dots, N_k \}.
\end{equation*}
The interior and boundary nodes are denoted by $\mathring{\mathcal N}_h = \mathcal N_h \cap \operatorname{int}\Omega_h$ and $\mathcal N_h^\partial = \mathcal N_h \cap \Gamma_h$, respectively.
We next define the global nodal basis functions $\phi_{\bm p} \, (\bm p \in \mathcal N_h)$ by
\begin{equation*}
	\phi_{\bm p}|_T = \begin{cases}
		0 & \text{ if } \bm p \notin T, \\
		\hat\phi_i \circ \bm F_T^{-1} & \text{ if $\bm p \in T$ and  $\bm p = \bm F_T(\hat{\bm a_i})$ with $\hat{\bm a_i} \in \Sigma_k$},
	\end{cases}
	\quad (\forall T \in \mathcal T_h)
\end{equation*}
which becomes continuous in $\overline\Omega_h$ thanks to the assumption on $\hat\Sigma_k$.
Then $\phi_{\bm p}(\bm q) = 1$ if $\bm p = \bm q$ and $\phi_{\bm p}(\bm q) = 0$ otherwise, for $\bm p, \bm q \in \mathcal N_h$.
We now set the $\mathbb P_k$-isoparametric finite element spaces by
\begin{equation*}
	V_h = \operatorname{span}\{\phi_{\bm p}\}_{\bm p \in \mathcal N_h} = \{ v_h \in C(\overline\Omega_h) \mid v_h\circ \bm F_T \in \mathbb P_k(\hat T) \; (\forall T \in \mathcal T_h) \}.
\end{equation*}
We see that $V_h \subset H^1(\Omega_h; \Gamma_h)$.
In particular, the restriction of $v_h \in V_h$ to $\Gamma_h$ is represented by $\mathbb P_k$-isoparametric finite element bases defined on $\Gamma_h$, that is,
\begin{equation*}
	(v_h \circ \bm F_{T_S})|_{\hat S} \in \mathbb P_k(\hat S) \quad (\forall S \in \mathcal S_h),
\end{equation*}
where $\hat S := \bm F_{T_S}^{-1}(S)$ denotes the pullback of the face $S$ in the reference coordinate (recall that $T_S$ is the element in $\mathcal T_h$ that contains $S$).

Noticing the chain rules $\nabla_{\bm x} = (\nabla_{\bm x}\bm F_T^{-1}) \nabla_{\hat{\bm x}}$, $\nabla_{\hat{\bm x}} = (\nabla_{\bm x}\bm F_T) \nabla_{\bm x}$ and the estimates given in \rref{rem: after regularity of order k}(v), we obtain the following estimates concerning the transformation between $\hat T$ and $T$:
\begin{prop}
	For $T \in \mathcal T_h$ and $v \in H^m(T)$ we have
	\begin{align*}
		\|\nabla_{\bm x}^m v\|_{L^2(T)} \le C h_T^{-m + d/2} \|\hat v\|_{H^m(\hat T)}, \qquad
		\|\nabla_{\hat{\bm x}}^m \hat v\|_{L^2(\hat T)} \le C h_T^{m - d/2} \|v\|_{H^m(T)},
	\end{align*}
	where $\hat v := v \circ \bm F_T \in H^m(\hat T)$.
\end{prop}
In particular, if $T \in \mathcal T_h$, $\bm p \in \mathcal N_h \cap T$, and $\bm p = \bm F_T(\hat{\bm a}_i)$, then
\begin{equation*}
	\|\nabla_{\bm x}^m \phi_{\bm p}\|_{L^2(T)} \le Ch_T^{-m + d/2} \Big( \sum_{l = 0}^m \|\nabla_{\hat{\bm x}}^l \hat\phi_{\bm p}\|_{L^2(\hat T)}^2 \Big)^{1/2}
		\le Ch_T^{-m + d/2},
\end{equation*}
where the quantities depending only on the reference element $\hat T$ have been combined into the generic constant.

To get an analogous estimate on the boundary $\Gamma_h$, we let $S$ be a curved $(d-1)$-face of $T \in \mathcal T_h$, i.e., $S = \bm F_T(\hat S)$ where $\hat S$ is a $(d-1)$-face of $\hat T$.
Then $\hat S$ is contained in some hyperplane $\hat x_d = \hat{\bm a}_{\hat S}' \cdot \hat{\bm x}' + \hat{\bm b}_{\hat S}$, and we get the following parametrization of $S$:
\begin{equation*}
	\bm F_S: \hat S' \to S; \quad \hat{\bm x}' \mapsto \bm F_T(\hat{\bm x}', \hat{\bm a}_{\hat S}' \cdot \hat{\bm x}' + \hat{\bm b}_{\hat S})  =: \bm F_T \circ \bm\Phi_{\hat S}(\hat{\bm x}'),
\end{equation*}
where $\hat S'$ is the projected image of $\hat S$ to the plane $\{x_d = 0\}$.
A similar parametrization can be obtained for the straight $(d-1)$-simplex $\tilde{\bm F}_T(\hat S) =: \tilde S$, which is denoted by $\tilde{\bm F}_S$ and is affine.

We see that the covariant and contravariant vectors $\tilde{\bm g}_\alpha, \tilde{\bm g}^\alpha$, and the covariant and contravariant components of metric tensors $\tilde G_{\alpha\beta}, \tilde G^{\alpha\beta}$ with respect to $\tilde S$ satisfies, for $\alpha, \beta = 1, \dots, d-1$,
\begin{align*}
	&|\tilde{\bm g}_\alpha| \le Ch_S, \quad |\tilde{\bm g}^\alpha| \le Ch_S^{-1}, \\
	&C_1 h_S^{d-1} \le \sqrt{\operatorname{det} \tilde G} = \frac{\operatorname{meas}_{d-1}(\tilde S)}{\operatorname{meas}_{d-1}(\hat S)} \le C_2 h_S^{d-1}, \quad
	|\tilde G_{\alpha\beta}| \le C h_S^2, \quad |\tilde G^{\alpha\beta}| \le C h_S^{-2},
\end{align*}
where $h_S := h_T$ and the regularity of the meshes has been used.
These vectors and components can also be defined for the curved simplex $S$, which are denoted by $\bar{\bm g}_\alpha, \bar{\bm g}^\alpha, \bar G_{\alpha\beta}, \bar G^{\alpha\beta}$.
Because $\bm F_S$ is a perturbation of $\tilde{\bm F}_S$, they satisfy the following estimates.
\begin{prop} \label{prop: transformation between S and Shat}
	(i) Let $m = 0, \dots, k$, and $\alpha, \beta = 1, \dots, d-1$.
	Then, for $S \in \mathcal S_h$ we have
	\begin{align*}
		&\|\nabla_{\hat{\bm x}'}^m \bar{\bm g}_\alpha\|_{L^\infty(\hat S')} \le Ch_S^{m+1}, \quad \|\nabla_{\hat{\bm x}'}^m \bar{\bm g}^\alpha\|_{L^\infty(\hat S')} \le Ch_S^{m-1}, \\
		&\|\nabla_{\hat{\bm x}'}^m \bar G_{\alpha\beta}\|_{L^\infty(\hat S')} \le C h_S^{m+2}, \quad \|\nabla_{\hat{\bm x}'}^m \bar G^{\alpha\beta}\|_{L^\infty(\hat S')} \le C h_S^{m-2}, \\
		&C_1 h_S^{(d-1)} \le \sqrt{\operatorname{det} \bar G} \le C_2 h_S^{(d-1)}.
	\end{align*}
	
	(ii) For $v \in H^m(S)$ we have
	\begin{align*}
		\|\nabla_{S}^m v\|_{L^2(S)} \le C h_S^{-m + (d-1)/2} \|v \circ \bm F_S\|_{H^m(\hat S')}, \qquad
		\|\nabla_{\hat{\bm x}'}^m (v \circ \bm F_S)\|_{L^2(\hat S')} \le C h_S^{m - (d-1)/2} \|v\|_{H^m(S)}.
	\end{align*}
\end{prop}
\begin{proof}
	(i) First let $m = 0$.
	Since $\bar{\bm g}_\alpha = (\frac{\partial \bm F_T}{\partial \hat x_\alpha} + \hat a_{\hat S\alpha}' \frac{\partial \bm F_T}{\partial \hat x_d})|_{\bm\Phi_{\hat S}}$, we have $\|\bar{\bm g}_\alpha\|_{L^\infty(\hat S')} \le Ch_S$, so that $\|\bar G_{\alpha\beta}\|_{L^\infty(\hat S')} \le C h_S^{2}$.
	By assumption \ref{H4}, we also get $\|\bar{\bm g}_\alpha - \tilde{\bm g}_\alpha\|_{L^\infty(\hat S')} \le Ch_S^2$ and $\|\bar G_{\alpha\beta} - \tilde G_{\alpha\beta}\|_{L^\infty(\hat S')} \le C h_S^3$, which allows us to bound $\operatorname{det} \bar G$ from above and below.
	This combined with the formula $\bar G^{-1} = (\operatorname{det} \bar G)^{-1} \operatorname{Cof} \bar G$ yields $\|\bar G^{\alpha\beta}\|_{L^\infty(\hat S')} \le C h_S^{-2}$, and, consequently, $\|\bar{\bm g}^\alpha\|_{L^\infty(\hat S')} \le Ch_S^{-1}$.
	
	The case $m \ge 1$ can be addressed by induction using assumption \ref{H6}.
	
	(ii) The first inequality is a result of $\nabla_S = \sum_{\alpha = 1}^{d-1} \bar{\bm g}^\alpha \frac{\partial}{\partial \hat x_\alpha}$ and (i).
	To show the second inequality, its inverted formula
	\begin{equation*}
		\frac{\partial}{\partial \hat x_\alpha} = \sum_{\beta = 1}^{d-1} \bar G_{\alpha\beta} (\bar{\bm g}^\beta \cdot \nabla_S)
	\end{equation*}
	is useful.
	We also notice the following for the case $m \ge 2$: even when $\nabla_S$ is acted on $\bar G_{\alpha\beta}, \bar{\bm g}^\beta$, or on their derivatives rather than on $v$, the $L^\infty$-bounds of them---in terms of the order of $h_S$---are the same as in the case where all the derivatives are applied to $v$.
	For example,
	\begin{equation*}
		\|\nabla_{\!S} \, \bar G_{\alpha\beta}\|_{L^\infty(\hat S')} = \Big\| \sum_{\alpha = 1}^{d-1} \bar{\bm g}^\alpha \frac{\partial \bar G_{\alpha\beta}}{\partial \hat x_\alpha} \Big\|_{L^\infty(\hat S')} \le Ch_S^{-1} \times Ch_S^3 = Ch_S^2,
	\end{equation*}
	which can be compared with $\|\bar G_{\alpha\beta}\|_{L^\infty(\hat S')} \le Ch_S^2$.
	Therefore, 
	\begin{align*}
		\|\nabla_{\hat{\bm x}'}^m (v \circ \bm F_S)\|_{L^2(\hat S')} &\le C (h_S^2 h_S^{-1})^m h_S^{(1-d)/2}
			\bigg[ \sum_{l = 0}^{k+1} \int_{\hat S'} \Big| \Big( \sum_{\alpha = 1}^{d-1} \bar{\bm g}^\alpha \frac{\partial}{\partial \hat x_\alpha} \Big)^l (v \circ \bm F_S) \Big|^2 \sqrt{\operatorname{det} \bar G} \, d\hat{\bm x}' \bigg]^{1/2} \\
		&= C h_S^{m - (d-1)/2} \|v\|_{H^{k+1}(S)},
	\end{align*}
	which is the desired estimate.
\end{proof}

In particular, if $\bm p \in \mathcal N_h \cap S$ and $\bm p = \bm F_T(\hat{\bm a}_i)$, we obtain
\begin{equation} \label{eq: nodal basis estimate on S}
	\|\nabla_S^m \phi_{\bm p}\|_{L^2(S)} \le Ch_S^{-m + (d-1)/2}.
\end{equation}

\subsection{Scott--Zhang interpolation operator}
We need the interpolation operator $\mathcal I_h$ introduced by \cite{ScoZha1990}, which is well-defined and stable in $H^1(\Omega_h)$.
We show that it is also stable in $H^1(\Gamma_h)$ on the boundary.
To each node $\bm p \in \mathcal N_h$ we assign $\sigma_{\bm p}$, which is either a $d$-curved simplex or $(d-1)$-curved simplex, in the following way:
\begin{itemize}
	\item If $\bm p \in \mathring{\mathcal N}_h$, we set $\sigma_{\bm p}$ to be one of the elements $T \in \mathcal T_h$ containing $\bm p$.
	\item If $\bm p \in \mathcal N_h^\partial$, we set $\sigma_{\bm p}$ to be one of the boundary elements $S \in \mathcal S_h$ containing $\bm p$.
\end{itemize}
For each $\bm p \in \mathcal N_h$, we see that $V_h|_{\sigma_{\bm p}}$ (the restrictions to $\sigma_{\bm p}$ of the functions in $V_h$) is a finite dimensional subspace of the Hilbert space $L^2(\sigma_{\bm p})$.
We denote by $\psi_{\bm q}$ the dual basis function corresponding to $\phi_{\bm p}$ with respect to $L^2(\sigma_{\bm p})$, that is, $\{\psi_{\bm q}\}_ {\bm q \in \mathcal N_h} \subset V_h$ is determined by
\begin{equation*}
	(\phi_{\bm p}, \psi_{\bm q})_{L^2(\sigma_{\bm p})} = 
	\begin{cases}
		1 & \text{if $\bm p = \bm q$}, \\
		0 & \text{otherwise},
	\end{cases}
	\qquad \forall \bm p \in \mathcal N_h.
\end{equation*}
The support of $\psi_{\bm p}$ is contained in a ``macro element'' of $\sigma_{\bm p}$.
In fact, depending on the cases $\sigma_{\bm p} = T \in \mathcal T_h$ and $\sigma_{\bm p} = S \in \mathcal S_h$, it holds that
\begin{align*}
	\operatorname{supp} \psi_{\bm p} \subset M_T &:= \bigcup \mathcal T_h(T), \quad \mathcal T_h(T) := \{T_1 \in \mathcal T_h \mid T_1 \cap T \neq \emptyset \}, \\
	\operatorname{supp} \psi_{\bm p} \subset M_S &:= \bigcup \mathcal S_h(S), \quad \mathcal S_h(S) := \{S_1 \in \mathcal S_h \mid S_1 \cap S \neq \emptyset \}.
\end{align*}

Now we define $\mathcal I_h : H^1(\Omega_h) \to V_h$ by
\begin{equation*}
	\mathcal I_h v = \sum_{\bm p \in \mathcal N_h} (v, \psi_{\bm p})_{L^2(\sigma_{\bm p})} \phi_{\bm p}.
\end{equation*}
By direct computation one can check $\mathcal I_h v_h = v_h$ for $v_h \in V_h$.
This invariance indeed holds at local level as shown in the lemma below.
To establish it, we first notice that $\mathcal I_h v$ in $T \in \mathcal T_h$ (resp.\ in $S \in \mathcal S_h$) is completely determined by $v$ in $M_T$ (resp.\ in $M_S$), which allows us to exploit the notation $(\mathcal I_h v)|_T$ for $v \in H^1(M_T)$ (resp.\ $(\mathcal I_h v)|_S$ for $v \in H^1(M_S)$).
\begin{rem}
	The choices of $\{\sigma_{\bm p}\}_{\bm p\in\mathcal N_h}$ and $\{\psi_{\bm p}\}_{\bm p \in \mathcal N_h}$ are not unique.
	Although the definition of $\mathcal I_h$ are dependent on those choices, the norm estimates below only depends on the shape-regularity constant and on a reference element.
\end{rem}

\begin{lem}
	Let $\bm p \in \mathcal N_h$ and $v \in H^1(\Omega_h)$.
	
	(i) If $\sigma_{\bm p} = T \in \mathcal T_h$, then
	\begin{equation*}
		\|\psi_{\bm p}\|_{L^\infty(T)} \le C h_T^{-d}.
	\end{equation*}
	Moreover, if $v \circ \bm F_{T_1} \in \mathbb P_k(\hat T)$ for $T_1 \in \mathcal T_h(T)$, then $(\mathcal I_h v)|_{T} = v|_{T}$.
	
	(ii) If $\sigma_{\bm p} = S \in \mathcal S_h$, then
	\begin{equation} \label{eq: Linfty norm of psip}
		\|\psi_{\bm p}\|_{L^\infty(S)} \le C h_S^{1-d}.
	\end{equation}
	Moreover, if $v \circ \bm F_{S_1} \in \mathbb P_k(\hat S')$ for $S_1 \in \mathcal S_h(S)$, then $(\mathcal I_h v)|_{S} = v|_{S}$.
\end{lem}
\begin{proof}
	We consider only case (ii); case (i) can be treated similarly.
	We can represent $\psi_{\bm p}$ as
	\begin{equation*}
		\psi_{\bm p} = \sum_{\bm q \in \mathcal N_h \cap M_S} C_{\bm p \bm q} \phi_{\bm q},
	\end{equation*}
	where $C = (C_{\bm p \bm q})$ is the inverse matrix of $A = ((\phi_{\bm p}, \phi_{\bm q})_{L^2(S)})$ (its dimension is supposed to be $D$).
	Note that each component of $A$ is bounded by $Ch_S^{d-1}$ and that $\operatorname{det} A \ge Ch_S^{D(d-1)}$.
	Therefore, each component of $C = (\operatorname{det} A)^{-1} \operatorname{Cof} A$ is bounded by $Ch_S^{(1-d)}$.
	This combined with $\|\phi_{\bm q}\|_{L^\infty(S)} \le C$ proves \eref{eq: Linfty norm of psip}.
	
	To show the second statement, observe that
	\begin{equation} \label{eq: Ih v restricted to S}
		(\mathcal I_h v)|_{S} = \sum_{\bm q \in \mathcal N_h} (v, \psi_{\bm q})_{L^2(\sigma_{\bm q})} \phi_{\bm q}|_S.
	\end{equation}
	However, $\phi_{\bm q}|_S$ is non-zero only if $\bm q \in S$, in which case $\sigma_{\bm q} \in \mathcal S_h(S)$.
	Therefore, $v|_{\sigma_{\bm q}}$ is represented as a linear combination of $\phi_{\bm s}|_{\sigma_{\bm q}} \, (\bm s \in \mathcal N_h \cap \sigma_{\bm q})$.
	This implies that \eref{eq: Ih v restricted to S} agrees with $v|_S$.
\end{proof}

Let us establish the stability of $\mathcal I_h$, which is divided into two lemmas and is proved in Appendix \ref{sec: stability of Ih}.

\begin{lem} \label{lem1: stability of Ih}
	Let $v \in H^1(\Omega_h; \Gamma_h)$, $T \in \mathcal T_h$, and $S \in \mathcal S_h$.
	Then for $m = 0, 1$ we have
	\begin{align}
		\|\nabla^m (\mathcal I_h v)\|_{L^2(T)} \le C \sum_{l=0}^1 h_T^{l-m} \sum_{T_1 \in \mathcal T_h(T)} \|\nabla^l v\|_{L^2(T_1)}, \notag \\
		\|\nabla_S^m (\mathcal I_h v)\|_{L^2(S)} \le C \sum_{l=0}^1 h_S^{l-m} \sum_{S_1 \in \mathcal S_h(S)} \|\nabla_{S_1}^l v\|_{L^2(S_1)}, \label{eq: local stability of Ihv on S}
	\end{align}
	where $\mathcal T_h(T) = \{T_1 \in \mathcal T_h \mid T_1 \cap T \neq \emptyset\}$ and $\mathcal S_h(S) = \{S_1 \in \mathcal S_h \mid S_1 \cap S \neq \emptyset\}$.
\end{lem}
\begin{lem} \label{lem2: stability of Ih}
	Under the same assumptions as in \lref{lem1: stability of Ih}, we have
	\begin{align}
		\|v - \mathcal I_h v\|_{H^m(T)} &\le Ch_T^{1-m} \sum_{T_1 \in \mathcal T_h(T)} \|v\|_{H^1(T_1)}, \notag \\
		\|v - \mathcal I_h v\|_{H^m(S)} &\le Ch_S^{1-m} \sum_{S_1 \in \mathcal S_h(S)} \|v\|_{H^1(S_1)}. \label{eq: local interpolation error on S}
	\end{align}
\end{lem}

Adding up \eref{eq: local interpolation error on S} for $S \in \mathcal S_h$ immediately leads to a global estimate (note that the regularity of the meshes implies $\sup_{S\in\mathcal S_h} \#\mathcal S_h(S) \le C$).
Together with an estimate in $\Omega_h$, which can be obtained in a similar manner, we state it as follows:
\begin{cor} \label{cor: H1 interpolation estimate}
	Let $m = 0, 1$ and $v \in H^1(\Omega_h; \Gamma_h)$. Then
	\begin{equation*}
		\|v - \mathcal I_h v\|_{H^m(\Omega_h)} \le Ch^{1-m} \|v\|_{H^1(\Omega_h)}, \qquad
		\|v - \mathcal I_h v\|_{H^m(\Gamma_h)} \le Ch^{1-m} \|v\|_{H^1(\Gamma_h)}.
	\end{equation*}
\end{cor}

\subsection{Interpolation error estimates}
First we recall the definition of the Lagrange interpolation operator and its estimates.
Define $\mathcal I_h^L : C(\overline\Omega_h) \to V_h$ by
\begin{equation*}
	\mathcal I_h^L v = \sum_{\bm p \in \mathcal N_h} v(\bm p) \phi_{\bm p}.
\end{equation*}
We allow the notation $(\mathcal I_h^L v)|_T$ if $v \in C(T)$, $T \in \mathcal T_h$, and $(\mathcal I_h^L v)|_S$ if $v \in C(S)$, $S \in \mathcal S_h$.

\begin{prop}
	Let $T \in \mathcal T_h$ and $S \in \mathcal S_h$.
	Assume $k + 1 > d/2$, so that $H^{k+1}(T) \hookrightarrow C(T)$ and $H^{k+1}(S) \hookrightarrow C(S)$ hold.
	Then, for $0\le m\le k+1$ we have
	\begin{align}
		\|\nabla^m(v - \mathcal I_h^L v)\|_{L^2(T)} &\le Ch_T^{k+1 - m} \|v\|_{H^{k+1}(T)} \qquad \forall v \in H^{k+1}(T), \label{eq: Lagrange interpolation error estimate in T} \\
		\|\nabla_S^m(v - \mathcal I_h^L v)\|_{L^2(S)} &\le Ch_S^{k+1 - m} \|v\|_{H^{k+1}(S)} \qquad \forall v \in H^{k+1}(S). \label{eq: Lagrange interpolation error estimate on S}
	\end{align}
\end{prop}
\begin{proof}
	By the Bramble--Hilbert theorem it holds that
	\begin{equation*}
		\|\nabla_{\hat{\bm x}'}^l [v \circ \bm F_S - (\mathcal I_h^L v) \circ \bm F_S)] \|_{L^2(\hat S')} \le C \|\nabla_{\hat{\bm x}'}^{k+1} (v \circ \bm F_S) \|_{L^2(\hat S')} \quad (l = 0, \dots, m),
	\end{equation*}
	where the constant $C$ depends only on $\hat S'$.
	This combined with \pref{prop: transformation between S and Shat}(ii) yields \eref{eq: Lagrange interpolation error estimate on S}.
	Estimate \eref{eq: Lagrange interpolation error estimate in T} is obtained similarly (or one can refer to \cite[Theorem 5]{CiaRav1972}).
\end{proof}
\begin{rem}
	(i) Adding up \eref{eq: Lagrange interpolation error estimate in T} for $T \in \mathcal T_h$ leads to the global estimate
	\begin{equation} \label{eq: global interpolation estimate in Omegah}
		\|v - \mathcal I_h^L v\|_{H^m(\Omega_h)} \le Ch^{k+1 - m} \|v\|_{H^{k+1}(\Omega_h)} \qquad \forall v \in H^{k+1}(\Omega_h) \quad (m = 0, 1).
	\end{equation}
	
	(ii) A corresponding global estimate on $\Gamma_h$ also holds; however, it is not useful for our purpose.
	To explain the reason, let us suppose $v \in H^m(\Omega; \Gamma)$ and extend it to some $\tilde v \in H^m(\mathbb R^d)$.
	Since we expect only $\tilde v|_{\Gamma_h} \in H^{m-1/2}(\Gamma_h)$ by the trace theorem, the direct interpolation $\mathcal I_h^L\tilde v$ may not have a good convergence property.
	To overcome this technical difficulty, we consider $\mathcal I_h^L (\tilde v \circ \bm\pi)$ instead in the theorem below, taking advantage of the fact that $v\circ\bm\pi$ is element-wisely as smooth on $\Gamma_h$ as $v$ is on $\Gamma$.
\end{rem}

\begin{thm} \label{thm: interpolation error estimate}
	Let $k + 1 > d/2$ and $m = 0,1$.
	For $v \in H^{k+1}(\Omega \cup \Gamma(\delta))$ satisfying $v|_\Gamma \in H^{k+1}(\Gamma)$ we have
	\begin{equation*}
		\| v - \mathcal I_h v\|_{H^{m}(\Omega_h; \Gamma_h)} \le Ch^{k+1-m} (\|v\|_{H^{k+1}(\Omega \cup \Gamma(\delta))} + \|v\|_{H^{k+1}(\Gamma)}).
	\end{equation*}
\end{thm}
\begin{proof}
	Let $\mathcal I$ denote the identity operator.
	Since $\mathcal I_h \mathcal I_h^L = \mathcal I_h^L$, one gets $\mathcal I - \mathcal I_h = (\mathcal I - \mathcal I_h)(\mathcal I - \mathcal I_h^L)$.
	Then it follows from \cref{cor: H1 interpolation estimate} and \eref{eq: global interpolation estimate in Omegah} that
	\begin{align*}
		\| v - \mathcal I_h  v\|_{H^m(\Omega_h)} &= \|(\mathcal I - \mathcal I_h)(v - \mathcal I_h^L v)\|_{H^m(\Omega_h)}
			\le Ch^{1-m} \| v - \mathcal I_h^L v\|_{H^1(\Omega_h)} \le Ch^{k+1-m} \|v\|_{H^{k+1}(\Omega_h)}.
	\end{align*}
	To consider the boundary estimate, observe that
	\begin{align*}
		 v - \mathcal I_h v = (\mathcal I - \mathcal I_h)(v - v \circ \bm\pi) + (\mathcal I - \mathcal I_h)(\mathcal I - \mathcal I_h^L) (v \circ \bm\pi)
			=: J_1 + J_2.
	\end{align*}
	By Corollaries \ref{cor: H1 interpolation estimate} and \ref{cor: u - u circ pi on Gammah},
	\begin{equation*}
		\|J_1\|_{H^m(\Gamma_h)} \le Ch^{1-m} \|v - v\circ\bm\pi\|_{H^1(\Gamma_h)} \le Ch^{k+1-m} \|v\|_{H^{\min\{k+1, 3\}}(\Omega \cup \Gamma(\delta))}.
	\end{equation*}
	From \cref{cor: H1 interpolation estimate}, \eref{eq: Lagrange interpolation error estimate on S}, and \eref{eq3: equivalence of surface integrals} we obtain
	\begin{align*}
		\|J_2\|_{H^m(\Gamma_h)} &\le Ch^{1-m} \|v\circ\bm\pi - \mathcal I_h^L(v\circ\bm\pi)\|_{H^1(\Gamma_h)} \le Ch^{k+1-m} \Big( \sum_{S \in \mathcal S_h} \|v\circ\bm\pi\|_{H^{k+1}(S)}^2 \Big)^{1/2} \\
			&\le Ch^{k+1-m} \Big( \sum_{S \in \mathcal S_h} \|v\|_{H^{k+1}(\bm\pi(S))}^2 \Big)^{1/2} = Ch^{k+1-m} \|v\|_{H^{k+1}(\Gamma)},
	\end{align*}
	where we have used \lref{lem: equivalence of Sobolev spaces on Gamma and Gammah}.
	Combining the estimates above proves the theorem.
\end{proof}

\section{Error estimates in an approximate domain} \label{sec: error estimate}
We continue to denote by $k \ge 1$ the order of the isoparametric finite element approximation throughout this and next sections.
\subsection{Finite element scheme based on extensions}
We recall that the weak formulation for \eref{eq1: g-robin}--\eref{eq2: g-robin} is given by \eref{eq: continuous problem}.
In order to define its finite element approximation, one needs counterparts to $f$ and $\tau$ given in $\Omega_h$ and $\Gamma_h$ respectively.
For this we will exploit extensions that preserves the smoothness as mentioned in Introduction.
Namely, if $f \in H^{k-1}(\Omega)$, one can choose some $\tilde f \in H^{k-1}(\mathbb R^d)$ such that $\|\tilde f\|_{H^{k-1}(\mathbb R^d)} \le C \|f\|_{H^{k-1}(\Omega)}$.
For $\tau$, we assume $\tau \in H^{k-1/2}(\Gamma)$ so that it admits an extension $\tilde\tau \in H^k(\mathbb R^d)$ such that $\|\tilde\tau\|_{H^{k}(\mathbb R^d)} \le C \|\tau\|_{H^{k-1/2}(\Gamma)}$
(the extension operator $\tilde\cdot$ has different meanings for $f$ and $\tau$, but there should be no fear of confusion).

The resulting discrete problem is to find $u_h \in V_h$ such that
\begin{equation} \label{eq: FE scheme}
	a_h(u_h, v_h) := (\nabla u_h, \nabla v_h)_{\Omega_h} + (u_h, v_h)_{\Gamma_h} + (\nabla_{\Gamma_h}u_h, \nabla_{\Gamma_h}v_h)_{\Gamma_h} = (\tilde f, v_h)_{\Omega_h} + (\tilde\tau, v_h)_{\Gamma_h} \qquad \forall v_h \in V_h.
\end{equation}
Because the bilinear form $a_h$ is uniformly coercive in $V_h$, i.e., $a_h(v_h, v_h) \ge C \|v_h\|_{H^1(\Omega_h; \Gamma_h)}^2$ for all $v_h \in V_h$ with $C$ independent of $h$, the existence and uniqueness of a solution $u_h$ is an immediate consequence of the Lax--Milgram theorem.

\subsection{$H^1$-error estimate}
We define the residual functionals for $v \in H^1(\Omega_h; \Gamma_h)$ by
\begin{align}
	R_u^1(v) &:= (-\Delta\tilde u - \tilde f, v)_{\Omega_h \setminus \Omega} + (\partial_{n_h}\tilde u - (\partial_n u)\circ\bm\pi, v)_{\Gamma_h} + (\tilde u - u\circ\bm\pi, v)_{\Gamma_h} + (\tau\circ\bm\pi - \tilde\tau, v)_{\Gamma_h}, \label{eq: R1u(v)} \\
	R_u^2(v) &:=  \big[ ( (\Delta_\Gamma u)\circ\bm\pi, v )_{\Gamma_h} + (\nabla_{\Gamma_h}(u\circ\bm\pi), \nabla_{\Gamma_h}v)_{\Gamma_h} \big]+ ( \nabla_{\Gamma_h}(\tilde u - u\circ\bm\pi), \nabla_{\Gamma_h} v_h )_{\Gamma_h}, \notag \\
	R_u(v) &:= R_u^1(v) + R_u^2(v), \notag
\end{align}
which completely vanish if we formally assume $\Omega_h = \Omega$.
Therefore, the residual terms above is considered to represent domain perturbation.
Let us state consistency error estimates, or, in other words, Galerkin orthogonality relation with domain perturbation terms.
\begin{prop}
	Assume that $f \in H^{k-1}(\Omega)$, $\tau \in H^{k-1/2}(\Gamma)$ if $k = 1, 2$, and that $f \in H^1(\Omega)$, $\tau \in H^{3/2}(\Gamma)$ if $k \ge 3$.
	Let $u$ and $u_h$ be the solutions of \eref{eq: continuous problem} and \eref{eq: FE scheme} respectively.
	Then we have
	\begin{equation} \label{eq: asymptotic Galerkin orthogonality}
		a_h(\tilde u - u_h, v_h) = R_u(v_h) \qquad \forall v_h \in V_h.
	\end{equation}
	Moreover, the following estimate holds:
	\begin{equation} \label{eq: 1st estimate of Ru(v)}
		|R_u(v)| \le C h^k (\|f\|_{H^{\min\{k-1, 1\}}(\Omega)} + \|\tau\|_{H^{\min\{k-1/2, 3/2\}}(\Gamma)}) \|v\|_{H^1(\Omega_h; \Gamma_h)} \qquad \forall v \in H^1(\Omega_h; \Gamma_h).
	\end{equation}
\end{prop}
\begin{proof}
	Equation \eref{eq: asymptotic Galerkin orthogonality} results from a direct computation as follows:
	\begin{align*}
		a_h(\tilde u - u_h, v_h) &= (\nabla(\tilde u - u_h), \nabla v_h)_{\Omega_h} + (\tilde u - u_h, v_h)_{\Gamma_h} + (\nabla_{\Gamma_h}(\tilde u - u_h), \nabla_{\Gamma_h} v_h)_{\Gamma_h} \\
			&= (-\Delta\tilde u , v_h)_{\Omega_h} + (\partial_{n_h}\tilde u + \tilde u, v_h)_{\Gamma_h} + (\nabla_{\Gamma_h}\tilde u, \nabla_{\Gamma_h} v_h)_{\Gamma_h} - (\tilde f, v_h)_{\Omega_h} - (\tilde\tau, v_h)_{\Gamma_h} \\
			&= (-\Delta\tilde u - \tilde f, v_h)_{\Omega_h \setminus \Omega} + (\partial_{n_h}\tilde u - (\partial_n u)\circ\bm\pi, v_h)_{\Gamma_h} + (\tilde u - u\circ\bm\pi, v_h)_{\Gamma_h} + (\tau\circ\bm\pi - \tilde\tau, v_h)_{\Gamma_h} \\
			&\hspace{1cm} + ((\Delta_\Gamma u)\circ\bm\pi, v_h)_{\Gamma_h} + (\nabla_{\Gamma_h} (u\circ\bm\pi), \nabla_{\Gamma_h} v_h)_{\Gamma_h} + (\nabla_{\Gamma_h}(\tilde u - u\circ\bm\pi), \nabla_{\Gamma_h} v_h)_{\Gamma_h} \\
			&= R_u^1(v_h) + R_u^2(v_h) = R_u(v_h).
	\end{align*}
	
	Let $C_{f,\tau}$ denote a generic constant multiplied by $\|f\|_{H^{\min\{k-1, 1\}}(\Omega)} + \|\tau\|_{H^{\min\{k-1/2, 3/2\}}(\Gamma)}$.
	We will make use of the regularity structure $\|u\|_{H^{k+1}(\Omega; \Gamma)} \le C (\|f\|_{H^{k-1}(\Omega)} + \|\tau\|_{H^{k-1}(\Gamma)})$ and the stability of extensions without further emphasis.
	Applying the boundary-skin estimate \eref{eq: RHS with Omegah minus Omega}, we obtain
	\begin{align*}
		|(-\Delta\tilde u - \tilde f, v)_{\Omega_h \setminus \Omega}| &\le
		\begin{cases}
			C (\|\Delta \tilde u\|_{L^2(\Omega_h)} + \|\tilde f\|_{L^2(\Omega_h)}) \cdot C\delta^{1/2} \|v\|_{H^1(\Omega_h)} & \quad (k = 1) \\
			C \delta^{1/2} (\|\tilde u\|_{H^3(\Omega_h)} + \|\tilde f\|_{H^1(\Omega_h)}) \cdot C \delta^{1/2} \|v\|_{H^1(\Omega_h)} & \quad (k \ge 2)
		\end{cases} \\
		&\le C_{f, \tau} h^k \|v\|_{H^1(\Omega_h)},
	\end{align*}
	where we have used $\delta = Ch^{k+1}$ and $h \le 1$.
	The second term of $R^1_u(v)$ is estimated as
	\begin{align*}
		|(\partial_{n_h}\tilde u - (\partial_n u)\circ\bm\pi, v)_{\Gamma_h}| &= \big|\big( \nabla\tilde u \cdot (\bm n_h - \bm n\circ\bm\pi), v \big)_{\Gamma_h} + \big((\nabla\tilde u - (\nabla u)\circ\bm\pi) \cdot \bm n\circ\bm\pi, v \big)_{\Gamma_h}\big| \\
			&\le C (h^k \|\nabla\tilde u\|_{L^2(\Gamma_h)} + \delta^{1/2} \|\nabla^2\tilde u\|_{L^2(\Gamma(\delta))}) \|v\|_{L^2(\Gamma_h)} \\
			&\le \begin{cases}
				C (h^k \|\tilde u\|_{H^2(\Omega_h)} + \delta^{1/2} \|\tilde u\|_{H^2(\Gamma(\delta))}) \|v\|_{H^1(\Omega_h)} & \quad (k = 1) \\
				C (h^k \|\tilde u\|_{H^2(\Omega_h)} + \delta \|\tilde u\|_{H^3(\Omega \cup \Gamma(\delta))}) \|v\|_{H^1(\Omega_h)} & \quad (k \ge 2)
			\end{cases} \\
			&\le C_{f, \tau} h^k \|v\|_{H^1(\Omega_h)},
	\end{align*}
	as a result of \eref{eq: n - nh}, \eref{eq3: boundary-skin estimates}, and \eref{eq2: boundary-skin estimates}.
	Similarly, the third term of $R^1_u(v)$ is bounded by
	\begin{align*}
		C \delta^{1/2} \|\nabla\tilde u\|_{L^2(\Gamma(\delta))} \|v_h\|_{L^2(\Gamma_h)} \le C_{f, \tau} h^k \|v_h\|_{H^1(\Omega_h)}.
	\end{align*}
	For the fourth term of $R^1_u(v)$, we need the regularity assumption $\tau \in H^{1/2}(\Gamma)$ for $k = 1$ and $\tau \in H^{3/2}(\Gamma)$ for $k \ge 2$ to ensure $\tilde\tau \in H^1(\mathbb R^d)$ and $\tilde\tau \in H^2(\mathbb R^d)$, respectively.
	Then $|(\tau\circ\bm\pi - \tilde\tau, v_h)_{\Gamma_h}|$ is bounded by
	\begin{align*}
		C \delta^{1/2} \|\nabla\tilde\tau\|_{L^2(\Gamma(\delta))} \|v_h\|_{L^2(\Gamma_h)} &\le
		\begin{cases}
			C \delta^{1/2} \|\nabla\tilde\tau\|_{L^2(\Gamma(\delta))} \|v_h\|_{H^1(\Omega_h)} \quad &(k = 1) \\
			C \delta \|\tilde\tau\|_{H^2(\Omega \cup \Gamma(\delta))} \|v_h\|_{H^1(\Omega_h)} &(k \ge 2)
		\end{cases} \\
		&\le C_{f, \tau} h^k \|v_h\|_{H^1(\Omega_h)}.
	\end{align*}
	
	For $R_u^2(v)$, we apply \lref{lem: error from integration by parts on Gammah} and \cref{cor: u - u circ pi on Gammah} to obtain
	\begin{align*}
		\big| ( (\Delta_\Gamma u)\circ\bm\pi, v )_{\Gamma_h} + (\nabla_{\Gamma_h}(u\circ\bm\pi), \nabla_{\Gamma_h}v)_{\Gamma_h} \big|
			&\le C\delta (\|u\|_{H^2(\Gamma)} \|v\|_{L^2(\Gamma_h)} + \|\nabla_\Gamma u\|_{L^2(\Gamma)} \|\nabla_{\Gamma_h} v\|_{L^2(\Gamma_h)}) \\
			&\le C_{f, \tau} h^k \|v_h\|_{H^1(\Gamma_h)}, \\
		\big| (\nabla_{\Gamma_h} (\tilde u - u\circ\bm\pi), \nabla_{\Gamma_h} v)_{\Gamma_h} \big| &\le Ch^k \|\tilde u\|_{H^{\min\{k+1, 3\}}(\Omega \cup \Gamma(\delta))} \|\nabla_{\Gamma_h} v\|_{L^2(\Gamma_h)} \le C_{f, \tau} h^k \|v_h\|_{H^1(\Gamma_h)}.
	\end{align*}
	Combining the estimates above all together concludes \eref{eq: 1st estimate of Ru(v)}.
\end{proof}
\begin{rem} \label{rem: in case of transformation less regularity for tau is OK}
	If the transformation $\tau\circ\bm\pi$ instead of the extension $\tilde\tau$ is employed in the FE scheme \eref{eq: FE scheme}, then assuming just $\tau \in H^{k-1}(\Gamma)$ is sufficient to get
	\begin{equation*}
		|R_u(v)| \le C h^k (\|f\|_{H^{\min\{k-1, 1\}}(\Omega)} + \|\tau\|_{H^{\min\{k-1, 1\}}(\Gamma)}) \|v\|_{H^1(\Omega_h; \Gamma_h)},
	\end{equation*}
	because the term involving $\tau$ in \eref{eq: R1u(v)} disappears.
\end{rem}

We are ready to state the $H^1$-error estimates.
\begin{thm} \label{thm: H1 error estimate}
	Let $k + 1 > d/2$.
	Assume that $f \in L^2(\Omega)$, $\tau \in H^{1/2}(\Gamma)$ for $k = 1$, that $f \in H^{1}(\Omega)$, $\tau \in H^{3/2}(\Gamma)$ for $k = 2$, and that $f \in H^{k-1}(\Omega)$, $\tau \in H^{k-1}(\Gamma)$ for $k \ge 3$.
	Then we have
	\begin{equation*}
		\|\tilde u - u_h\|_{H^1(\Omega_h; \Gamma_h)} \le C h^k (\|f\|_{H^{k-1}(\Omega)} + \|\tau\|_{H^{\max\{k-1, \min\{k-1/2,3/2\}\}}(\Gamma)}),
	\end{equation*}
	where $u$ and $u_h$ are the solutions of \eref{eq: continuous problem} and \eref{eq: FE scheme} respectively.
\end{thm}
\begin{proof}
	To save the space we introduce the notation $C_{f, \tau} := C (\|f\|_{H^{k-1}(\Omega)} + \|\tau\|_{H^{\max\{k-1, \min\{k-1/2,3/2\}\}}(\Gamma)})$.
	It follows from the uniform coercivity of $a_h$ and \eref{eq: asymptotic Galerkin orthogonality} that
	\begin{align*}
		C\|\tilde u - u_h\|_{H^1(\Omega_h; \Gamma_h)}^2 \le a_h(\tilde u - u_h, \tilde u - u_h) = a_h(\tilde u - u_h, \tilde u - \mathcal I_h \tilde u) + R_u(\mathcal I_h \tilde u - u_h).
	\end{align*}
	In view of \tref{thm: interpolation error estimate}, the first term in the right-hand side is bounded by
	\begin{align*}
		\|\tilde u - u_h\|_{H^1(\Omega_h; \Gamma_h)} \|\tilde u - \mathcal I_h \tilde u\|_{H^1(\Omega_h; \Gamma_h)} &\le Ch^k (\|\tilde u\|_{H^{k+1}(\Omega \cup \Gamma(\delta))} + \|u\|_{H^{k+1}(\Gamma)}) \|\tilde u - u_h\|_{H^1(\Omega_h; \Gamma_h)} \\
			&\le C_{f, \tau} h^k \|\tilde u - u_h\|_{H^1(\Omega_h; \Gamma_h)}
	\end{align*}
	as a result of the regularity of $u$ and the stability of extensions.
	Estimate \eref{eq: 1st estimate of Ru(v)} applied to $R_u(\mathcal I_h \tilde u - u_h)$ combined again with \tref{thm: interpolation error estimate} gives the upper bound of the second term as
	\begin{align*}
		C_{f, \tau} h^k \|\mathcal I_h \tilde u - u_h\|_{H^1(\Omega_h; \Gamma_h)} \le C_{f, \tau} h^k \|\tilde u - u_h\|_{H^1(\Omega_h; \Gamma_h)} + (C_{f, \tau}h^{k})^2.
	\end{align*}
	Consequently,
	\begin{equation*}
		C\|\tilde u - u_h\|_{H^1(\Omega_h; \Gamma_h)}^2 \le C_{f, \tau} h^k \|\tilde u - u_h\|_{H^1(\Omega_h; \Gamma_h)} + (C_{f, \tau} h^k)^2,
	\end{equation*}
	which after an absorbing argument proves the theorem.
\end{proof}

\subsection{$L^2$-error estimate}
Let $\varphi\in L^2(\Omega_h), \psi \in L^2(\Gamma_h)$ be arbitrary such that $\|\varphi\|_{L^2(\Omega_h)} = \|\psi\|_{L^2(\Gamma_h)} = 1$.
We define $w \in H^2(\Omega; \Gamma)$ as the solution of the dual problem introduced as follows:
\begin{equation} \label{eq: dual problem}
	-\Delta w = \varphi \quad\text{in }\;\Omega, \qquad \textstyle\frac{\partial w}{\partial n} + w - \Delta_{\Gamma} w = \psi\circ\bm\pi^* \quad\text{on }\; \Gamma,
\end{equation}
where $\varphi$ is extended to $\mathbb R^d \setminus \Omega_h$ by 0.
For $v \in H^1(\Omega_h; \Gamma_h)$ we define residual functionals w.r.t.\ $w$ by
\begin{align*}
	R_w^1(v) &:= (v, -\Delta\tilde w - \varphi)_{\Omega_h\setminus\Omega} + (v, \partial_{n_h}\tilde w - (\partial_n w) \circ \pi)_{\Gamma_h} + (v, \tilde w - w\circ\bm\pi)_{\Gamma_h}, \\
	R_w^2(v) &:= \big[ (v, (\Delta_\Gamma w)\circ\bm\pi)_{\Gamma_h} + ( \nabla_{\Gamma_h} v, \nabla_{\Gamma_h}(w\circ\bm\pi))_{\Gamma_h} \big] + ( \nabla_{\Gamma_h} v, \nabla_{\Gamma_h}(\tilde w - w\circ\bm\pi) )_{\Gamma_h} \\
	R_w(v) &:= R_w^1(v) + R_w^2(v).
\end{align*}

\begin{lem}
	Let $k \ge 1$, $v \in H^1(\Omega_h; \Gamma_h)$, and $w$ be as above.
	Then we have
	\begin{equation} \label{eq: L2 duality representation for phi and psi}
		(v, \varphi)_{\Omega_h} + (v, \psi)_{\Gamma_h} = a_h(v, \tilde w) - R_w(v).
	\end{equation}
	Moreover, the following estimate holds:
	\begin{equation} \label{eq: estimate of Rw(v)}
		|R_w(v)| \le Ch \|w\|_{H^2(\Omega; \Gamma)} \|v\|_{H^1(\Omega_h; \Gamma_h)}.
	\end{equation}
\end{lem}
\begin{proof}
	A direct computation shows
	\begin{align*}
		a_h(v, \tilde w) &= (\nabla v, \nabla\tilde w)_{\Omega_h} + (v, \tilde w)_{\Gamma_h} + (\nabla_{\Gamma_h} v, \nabla_{\Gamma_h}\tilde w)_{\Gamma_h} \\
			&= (v, -\Delta\tilde w)_{\Omega_h} + (v, \partial_{n_h}\tilde w)_{\Gamma_h} + (v, \tilde w)_{\Gamma_h} + (\nabla_{\Gamma_h} v, \nabla_{\Gamma_h}\tilde w)_{\Gamma_h} \\
			&= [(v, \varphi)_{\Omega_h} + (v, -\Delta\tilde w - \varphi)_{\Omega_h\setminus\Omega}] + (v, \partial_{n_h}\tilde w - (\partial_n w) \circ \pi)_{\Gamma_h} \\
			&\qquad	+ (v, \tilde w - w\circ\bm\pi)_{\Gamma_h} + (v, (\Delta_\Gamma w)\circ\bm\pi)_{\Gamma_h} + (\nabla_{\Gamma_h} v, \nabla_{\Gamma_h}\tilde w)_{\Gamma_h} + (v, \psi)_{\Gamma_h} \\
			&= (v, \varphi)_{\Omega_h} + (v, \psi)_{\Gamma_h} + R_w^1(v) + R_w^2(v),
	\end{align*}
	which is \eref{eq: L2 duality representation for phi and psi}.
	Estimate \eref{eq: estimate of Rw(v)} is obtained by almost the same manner as \eref{eq: 1st estimate of Ru(v)} for $k = 1$.
	The only difference is that no domain perturbation term involving $\psi$ appears this time (cf.\ \rref{rem: in case of transformation less regularity for tau is OK}).
\end{proof}

Next we show that $R_u(v)$ admits another equivalent representation if $v \in H^2(\Omega \cup \Gamma(\delta))$ and $v|_{\Gamma} \in H^2(\Gamma)$.
We make use of the integration by parts formula
\begin{equation} \label{eq: integration by parts in symmetric difference}
	(\Delta u, v)_{\Omega_h\triangle\Omega}' + (\nabla u, \nabla v)_{\Omega_h\triangle\Omega}' =  (\partial_{n_h}u, v)_{\Gamma_h} - (\partial_n u, v)_{\Gamma},
\end{equation}
where $(u, v)_{\Omega_h\triangle\Omega}' := (u, v)_{\Omega_h \setminus \Omega} - (u, v)_{\Omega \setminus \Omega_h}$.
\begin{prop}
	Let $k \ge 1$, $f \in H^1(\Omega)$, $\tau \in H^{3/2}(\Gamma)$.
	Assume that $u \in H^{\min\{k+1, 3\}}(\Omega; \Gamma)$ be the solution of \eref{eq: continuous problem}.
	Then, for $v \in H^2(\Omega \cup \Gamma(\delta))$ we have
	\begin{equation} \label{eq: another form of Ru(v)}
		R_u(v) = -(\tilde f, v)_{\Omega_h \triangle \Omega}' + (\tilde u - \tilde\tau, v)_{\Gamma_h \cup \Gamma}' + (\nabla\tilde u, \nabla v)_{\Omega_h\triangle\Omega}' + (\nabla_{\Gamma_h}\tilde u, \nabla_{\Gamma_h} v)_{\Gamma_h} - (\nabla_\Gamma u, \nabla_\Gamma v)_\Gamma,
	\end{equation}
	where $(u, v)_{\Gamma_h \cup \Gamma}' := (u, v)_{\Gamma_h} - (u, v)_\Gamma$.
	If in addition $v|_\Gamma \in H^2(\Gamma)$, the following estimate holds:
	\begin{equation} \label{eq: 2nd estimate of Ru(v)}
		|R_u(v)| \le C\delta (\|f\|_{H^1(\Omega)} + \|\tau\|_{H^{3/2}(\Gamma)}) (\|v\|_{H^2(\Omega \cup \Gamma(\delta))} + \|v\|_{H^2(\Gamma)}).
	\end{equation}
\end{prop}
\begin{proof}
	Since $-\Delta u = f$ in $\Omega$ and $-\partial_n u - u + \tau + \Delta_\Gamma u = 0$ on $\Gamma$, it follows from \eref{eq: integration by parts in symmetric difference} that
	\begin{align*}
		R_u(v) &= (-\Delta\tilde u  - \tilde f, v)_{\Omega_h\triangle\Omega}' + (\partial_{n_h}\tilde u, v)_{\Gamma_h} + (\tilde u - \tilde\tau, v)_{\Gamma_h} + (\nabla_{\Gamma_h}\tilde u, \nabla_{\Gamma_h} v)_{\Gamma_h} \\
			&= - (\tilde f, v)_{\Omega_h\triangle\Omega}' + (\nabla\tilde u, \nabla v)_{\Omega_h\triangle\Omega}' + (\partial_n u, v)_{\Gamma} + (\tilde u - \tilde\tau, v)_{\Gamma_h} + (\nabla_{\Gamma_h}\tilde u, \nabla_{\Gamma_h} v)_{\Gamma_h} \\
			&= - (\tilde f, v)_{\Omega_h\triangle\Omega}' + (\nabla\tilde u, \nabla v)_{\Omega_h\triangle\Omega}' + (-u + \tau + \Delta_\Gamma u, v)_{\Gamma} + (\tilde u - \tilde\tau, v)_{\Gamma_h} + (\nabla_{\Gamma_h}\tilde u, \nabla_{\Gamma_h} v)_{\Gamma_h},
	\end{align*}
	which after the integration by parts on $\Gamma$ yields \eref{eq: another form of Ru(v)}.
	
	By the boundary-skin estimates, the regularity structure $\|u\|_{H^2(\Omega; \Gamma)} \le C(\|f\|_{L^2(\Omega)} + \|\tau\|_{L^2(\Gamma)})$, and the stability of extensions, the first three terms on the right-hand side of \eref{eq: another form of Ru(v)} is bounded as follows:
	\begin{align*}
		|(\tilde f, v)_{\Omega_h\triangle\Omega}'| &\le \|\tilde f\|_{L^2(\Gamma(\delta))} \|v\|_{L^2(\Gamma(\delta))} \le C \delta \|f\|_{H^1(\Omega)} \|v\|_{H^1(\Omega \cup \Gamma(\delta))}, \\
		|(\tilde u - \tilde\tau, v)_{\Gamma_h \cup \Gamma}'| &\le C \delta \|\tilde u - \tilde\tau\|_{H^2(\Omega \cup \Gamma(\delta))} \|v\|_{H^2(\Omega \cup \Gamma(\delta))} \le C \delta (\|f\|_{L^2(\Omega)} + \|\tau\|_{H^{3/2}(\Gamma)}) \|v\|_{H^2(\Omega \cup \Gamma(\delta))}, \\
		|(\nabla\tilde u, \nabla v)_{\Omega_h\triangle\Omega}'| &\le C \delta \|\nabla\tilde u\|_{H^1(\Omega \cup \Gamma(\delta))} \|\nabla v\|_{H^1(\Omega \cup \Gamma(\delta))} \le C \delta (\|f\|_{L^2(\Omega)} + \|\tau\|_{L^2(\Gamma)}) \|v\|_{H^2(\Omega \cup \Gamma(\delta))}.
	\end{align*}
	For the fourth and fifth terms of \eref{eq: another form of Ru(v)}, we start from the obvious equality
	\begin{align*}
		&(\nabla_{\Gamma_h}\tilde u, \nabla_{\Gamma_h} v)_{\Gamma_h} - (\nabla_\Gamma u, \nabla_\Gamma v)_\Gamma \\
		= \; &( \nabla_{\Gamma_h}(\tilde u - u\circ\bm\pi), \nabla_{\Gamma_h}(v - v\circ\bm\pi) )_{\Gamma_h} + ( \nabla_{\Gamma_h}(u\circ\bm\pi), \nabla_{\Gamma_h}(v - v\circ\bm\pi) )_{\Gamma_h} \\
		&\qquad + ( \nabla_{\Gamma_h}(\tilde u - u\circ\bm\pi), \nabla_{\Gamma_h}(v\circ\bm\pi) )_{\Gamma_h} + \big[ ( \nabla_{\Gamma_h}(u\circ\bm\pi), \nabla_{\Gamma_h}(v\circ\bm\pi) )_{\Gamma_h}  - (\nabla_\Gamma u, \nabla_\Gamma v)_\Gamma \big] \\
		=: \; &I_1 + I_2 + I_3 + I_4.
	\end{align*}
	By \cref{cor: u - u circ pi on Gammah}, $|I_1| \le Ch^{2k} \|\tilde u\|_{H^{\min\{k+1, 3\}}(\Omega \cup \Gamma(\delta))} \|v\|_{H^{\min\{k+1, 3\}}(\Omega \cup \Gamma(\delta))}$ (note that $h^{2k} \le C\delta$).
	From \lref{lem: inner product between nabla(u - u circ pi) and nabla v circ pi} we have
	\begin{equation*}
		|I_2| \le C \delta \|u\|_{H^2(\Gamma)} \|v\|_{H^2(\Omega \cup \Gamma(\delta))}, \qquad |I_3| \le C \delta \|\tilde u\|_{H^2(\Omega \cup \Gamma(\delta))} \|v\|_{H^2(\Gamma)}.
	\end{equation*}
	Finally, $|I_4| \le C \delta \|u\|_{H^1(\Gamma)} \|v\|_{H^1(\Gamma)}$ by \eref{eq: error between nabla Gammah and nabla Gamma}.
	Combining the estimates above concludes \eref{eq: 2nd estimate of Ru(v)}.
\end{proof}
\begin{rem}
	We need $f \in H^1(\Omega)$ and $\tau \in H^{3/2}(\Gamma)$ even for $k = 1$.
\end{rem}

We are in the position to state the $L^2$-error estimate in $\Omega_h$ and on $\Gamma_h$.
\begin{thm} \label{thm: L2 error estimate}
	Let $k + 1 > d/2$.
	Assume that $f \in H^1(\Omega)$, $\tau \in H^{3/2}(\Gamma)$ for $k = 1, 2$ and that $f \in H^{k-1}(\Omega)$, $\tau \in H^{k-1}(\Gamma)$ for $k \ge 3$.
	Then we have
	\begin{equation*}
		\|\tilde u - u_h\|_{L^2(\Omega_h; \Gamma_h)} \le C_{f, \tau}h^{k+1},
	\end{equation*}
	where $C_{f, \tau} := C (\|f\|_{H^{\max\{k-1, 1\}}(\Omega)} + \|\tau\|_{H^{\max\{k-1, 3/2\}}(\Gamma)})$.
\end{thm}
\begin{proof}
	We consider the solution $w$ of \eref{eq: dual problem} obtained from the following choices of $\varphi$ and $\psi$:
	\begin{equation*}
		\varphi = \frac{\tilde u - u_h}{\|\tilde u - u_h\|_{L^2(\Omega)}}, \qquad \psi = \frac{\tilde u - u_h}{\|\tilde u - u_h\|_{L^2(\Gamma)}}.
	\end{equation*}
	Taking then $v = \tilde u - u_h$ in \eref{eq: L2 duality representation for phi and psi} and using \eref{eq: asymptotic Galerkin orthogonality}, we obtain
	\begin{align*}
		\|\tilde u - u_h\|_{L^2(\Omega_h; \Gamma_h)} &= a_h(\tilde u - u_h, \tilde w) - R_w(\tilde u - u_h) \\
			&= a_h(\tilde u - u_h, \tilde w - w_h) - R_u(\tilde w - w_h) - R_w(\tilde u - u_h) + R_u(\tilde w),
	\end{align*}
	where we set $w_h := \mathcal I_h \tilde w$.
	Since $\|\tilde u - u_h\|_{H^1(\Omega_h; \Gamma_h)} \le C_{f, \tau} h^{k}$ by \tref{thm: H1 error estimate} and $\|w\|_{H^2(\Omega; \Gamma)} \le C$, we find from \tref{thm: interpolation error estimate} and the residual estimates \eref{eq: 1st estimate of Ru(v)}, \eref{eq: estimate of Rw(v)}, and \eref{eq: 2nd estimate of Ru(v)} that
	\begin{align*}
		|a_h(\tilde u - u_h, \tilde w - w_h)| &\le C\|\tilde u - u_h\|_{H^1(\Omega_h; \Gamma_h)} \|\tilde w - w_h\|_{H^1(\Omega_h; \Gamma_h)} \le C_{f, \tau} h^{k+1}, \\
		| R_u(\tilde w - w_h) - R_w(\tilde u - u_h) | &\le C_{f, \tau} h^k \|\tilde w - w_h\|_{H^1(\Omega_h; \Gamma_h)} + Ch \|\tilde u - u_h\|_{H^1(\Omega_h; \Gamma_h)} \le C_{f, \tau} h^{k+1}, \\
		| R_u(\tilde w) | &\le C_{f, \tau} \delta (\|\tilde w\|_{H^2(\Omega \cup \Gamma(\delta))} + \|w\|_{H^2(\Gamma)}) \le C_{f, \tau}h^{k+1},
	\end{align*}
	where the stability of extensions has been used.
	This proves the theorem.
\end{proof}

\section{Numerical example}
Let $\Omega = \{(x, y) \in \mathbb R^2 \mid x^2 + y^2 = 1\}$ be the unit disk (thus $\Gamma$ is the unit circle) and set the exact solution to be
\begin{equation*}
	u(x, y) = 10x^2 y.
\end{equation*}
With the linear finite element method, i.e., $k = 1$, we compute approximate solutions using the software \texttt{FreeFEM}.
The surface gradient $\nabla_{\Gamma_h} u_h$ is computed by
\begin{equation*}
	\nabla_{\Gamma_h} u_h = (I - \bm n_h \otimes \bm n_h) \nabla u_h \quad\text{on }\; \Gamma_h.
\end{equation*}
The errors are computed by interpolating the exact solution to the quadratic finite element spaces.
The results are reported in Table \ref{tab: error}, where $N$ denotes the number of nodes on the boundary.
We see that the $H^1(\Omega_h; \Gamma_h)$- and $L^2(\Omega_h; \Gamma_h)$-errors behave as $O(h)$ and $O(h^2)$ respectively, which is consistent with the theoretical results established in Theorems \ref{thm: H1 error estimate} and \ref{thm: L2 error estimate}.

\begin{table}[htbp]
	\small
	\centering
	\caption{behavior of the error in $H^1(\Omega_h)$, $H^1(\Gamma_h)$, $L^2(\Omega_h)$, and $L^2(\Gamma_h)$}
	\label{tab: error}
	\begin{tabular}{cccccc}
		$N$ & $h$ & $\|\nabla(u - u_h)\|_{L^2(\Omega_h)}$ & $\|\nabla_{\Gamma_h} (u - u_h)\|_{L^2(\Gamma_h)}$ & $\|u - u_h\|_{L^2(\Omega_h)}$ & $\|u - u_h\|_{L^2(\Gamma_h)}$ \\
		\hline
		32 & 0.293 & 1.84 & 1.58 & 7.23E-2 & 0.129 \\
		64 & 0.161 & 0.93 & 0.794 & 1.81E-2 & 3.27E-2 \\
		128 & 9.44E-2 & 0.460 & 0.397 & 4.49E-3 & 8.11E-3 \\
		256 & 4.26E-2 & 0.229 & 0.199 & 1.09E-3 & 2.03E-3
	\end{tabular}
\end{table}



\begin{thebibliography}{10}

\bibitem{BaEl88}
{\sc J.~W. Barrett and C.~M. Elliott}, {\em Finite-element approximation of
  elliptic equations with a {N}eumann or {R}obin condition on a curved
  boundary}, IMA J. Numer.\ Anal., 8 (1988), pp.~321--342.

\bibitem{Ber1989}
{\sc C.~Bernardi}, {\em Optimal finite-element interpolation on curved
  domains}, SIAM J. Numer.\ Anal., 26 (1989), pp.~1212--1240.

\bibitem{BerGir1998}
{\sc C.~Bernardi and V.~Girault}, {\em A local regularization operator for
  triangular and quadrilateral finite elements}, SIAM J. Numer.\ Anal., 35
  (1998), pp.~1893--1916.

\bibitem{ChiSai2023}
{\sc Y.~Chiba and N.~Saito}, {\em Nitsche's method for a {R}obin boundary value
  problem in a smooth domain}, Numer.\ Methods Partial Differential Eq., 39
  (2023), pp.~4126--4144.

\bibitem{Cia78}
{\sc P.~G. Ciarlet}, {\em The Finite Element Method for Elliptic Problems},
  SIAM, 1978.

\bibitem{CiaRav1972}
{\sc P.~G. Ciarlet and P.-A. Raviart}, {\em Interpolation theory over curved
  elements, with applications to finite element methods}, Comput.\ Math.\
  Appl.\ Mech.\ Engrg., 1 (1972), pp.~217--249.

\bibitem{CDQ2014}
{\sc C.~M. Colciago, S.~Deparis, and A.~Quarteroni}, {\em Comparisons between
  reduced order models and full 3{D} models for fluid-structure interaction
  problems in haemodynamics}, J. Comput.\ Appl.\ Math., 265 (2014),
  pp.~120--138.

\bibitem{DeZo11}
{\sc M.~C. Delfour and J.-P. Zol\'esio}, {\em Shapes and Geometries---Metrics,
  Analysis, Differential Calculus, and Optimization}, SIAM, 2nd~ed., 2011.

\bibitem{Ede2021}
{\sc D.~Edelmann}, {\em Isoparametric finite element analysis of a generalized
  {R}obin boundary value problem on curved domains}, SMAI J. Comput.\ Math., 7
  (2021), pp.~57--73.

\bibitem{EllRan2013}
{\sc C.~M. Elliott and T.~Ranner}, {\em Finite element analysis for a coupled
  bulk-surface partial differential equations}, IMA J. Numer.\ Anal., 33
  (2013), pp.~377--402.

\bibitem{GiTr98}
{\sc D.~Gilbarg and N.~S. Trudinger}, {\em Elliptic Partial Differential
  Equations of Second Order}, Springer, 1998.

\bibitem{KCDQ2015}
{\sc T.~Kashiwabara, C.~M. Colciago, L.~Ded\`e, and A.~Quarteroni}, {\em
  Well-posedness, regularity, and convergence analysis of the finite element
  approximation of a generalized {R}obin boundary value problem}, SIAM J.
  Numer.\ Anal., 53 (2015), pp.~105--126.

\bibitem{KasKem2020a}
{\sc T.~Kashiwabara and T.~Kemmochi}, {\em Pointwise error estimates of linear
  finite element method for {N}eumann boundary value problems in a smooth
  domain}, Numer.\ Math., 144 (2020), pp.~553--584.

\bibitem{KOZ16}
{\sc T.~Kashiwabara, I.~Oikawa, and G.~Zhou}, {\em Penalty method with
  {P}1/{P}1 finite element approximation for the stokes equations under the
  slip boundary condition}, Numer.\ Math., 134 (2016), pp.~705--740.

\bibitem{KovLub2017}
{\sc B.~Kov\'acs and C.~Lubich}, {\em Numerical analysis of parabolic problems
  with dynamic boundary conditions}, IMA J. Numer.\ Anal., 37 (2017),
  pp.~1--39.

\bibitem{Len86}
{\sc M.~Lenoir}, {\em Optimal isoparametric finite elements and error estimates
  for domains involving curved boundaries}, SIAM J. Numer.\ Anal., 21 (1986),
  pp.~562--580.

\bibitem{Ric2017}
{\sc T.~Richter}, {\em Fluid-structure Interactions---Moldels, Analysis and
  Finite Elements}, Springer, 2017.

\bibitem{ScoZha1990}
{\sc L.~R. Scott and S.~Zhang}, {\em Finite element interpolation of nonsmooth
  functions satisfying boundary conditions}, Math.\ Comp., 54 (1990),
  pp.~483--493.

\end{thebibliography}

\appendix
\section{Proof of \lref{lem: nablaGammah(v - vpi)}} \label{apx: nablaGammah(v - vpi)}
We consider the case $p < \infty$ only because $p = \infty$ can be addressed by an obvious modification.
We start from the local coordinate representation
\begin{equation} \label{eq1: proof of nabla (v - v circ pi)}
	\int_S |\nabla_{\Gamma_h} (v - v\circ\bm\pi)|^p \, d\gamma_h = \int_{S'} \Big| \sum_{\alpha} \bm g_h^\alpha \partial_\alpha \big[ v(\bm\Phi_h(\bm z')) - v(\bm\Phi(\bm z')) \big] \Big|^p \sqrt{\operatorname{det}G_h} \, d\bm z'.
\end{equation}
Since $\bm\Psi(\bm z', t) = \bm\Phi(\bm z') + t \bm n(\bm\Phi(\bm z'))$, one has
\begin{equation*}
	v(\bm\Phi_h(\bm z')) - v(\bm\Phi(\bm z')) = \int_0^{t^*(\bm\Phi(\bm z'))} \bm n(\bm\Phi(\bm z')) \cdot (\nabla v) \circ \bm\Psi(\bm z', t) \, dt \quad (\bm z' \in S').
\end{equation*}
Consequently,
\begin{align*}
	\sum_{\alpha} \bm g_h^\alpha \partial_\alpha \big[ v(\bm\Phi_h(\bm z')) - v(\bm\Phi(\bm z')) \big]
		&= \sum_{\alpha=1}^{d-1} \bm g_h^\alpha \, \partial_\alpha(t^* \circ \bm\Phi) \, \big[ (\bm n\circ\bm\Phi) \cdot ((\nabla v) \circ \bm\Phi_h) \big] \\
		&+ \sum_{\alpha=1}^{d-1} \bm g_h^\alpha \int_0^{t^*\circ\bm\Phi} \Big( \partial_\alpha(\bm n\circ\bm\Phi) \cdot (\nabla v)\circ\bm\Psi + \bm n\circ\bm\Phi \cdot (\nabla^2 v)\circ\bm\Psi \cdot \partial_\alpha\bm\Psi \Big) \, dt,
\end{align*}
where $\bm a \cdot A \cdot \bm b$ means ${}^t\!\bm a A \bm b$ for vectors $\bm a, \bm b$ and a matrix $A$.

For the first term, since $|\partial_\alpha (t^*\circ\bm\Phi)| \le Ch_S^k$ by \eref{eq: global t*},
\begin{equation} \label{eq2: proof of nabla (v - v circ pi)}
	\int_{S'} \Big| \sum_{\alpha=1}^{d-1} \bm g_h^\alpha \, \partial_\alpha(t^* \circ \bm\Phi) \, \big[ (\bm n\circ\bm\Phi) \cdot ((\nabla v) \circ \bm\Phi_h) \big] \Big|^p \sqrt{\operatorname{det}G_h} \, d\bm z' \le Ch_S^{kp} \|\nabla u\|_{L^p(S)}^p.
\end{equation}
For the second term, since $|\bm g_h^\alpha| \le C, |\partial_\alpha(\bm n\circ\bm\Phi)| \le C$, $|t^*\circ\bm\Phi| \le C\delta_S$, and $|\partial_\alpha\bm\Psi| = |\partial_\alpha\bm\Phi + t \partial_\alpha(\bm n\circ\bm\Phi)| \le C$, we obtain
\begin{align}
	&\int_{S'} \bigg| \sum_{\alpha=1}^{d-1} \bm g_h^\alpha \int_0^{t^*\circ\bm\Phi} \Big( \partial_\alpha(\bm n\circ\bm\Phi) \cdot (\nabla v)\circ\bm\Psi + \bm n\circ\bm\Phi \cdot (\nabla^2 v)\circ\bm\Psi \cdot \partial_\alpha\bm\Psi \Big) \, dt \bigg|^p \sqrt{\operatorname{det}G_h} \, d\bm z' \notag \\
	\le \; &C \int_{S'} \bigg| \int_{-\delta_S}^{\delta_S} \Big( |(\nabla v)\circ\bm\Psi| + |(\nabla^2 v)\circ\bm\Psi| \Big) \, dt \bigg|^p \, d\bm z' \notag \\
	\le \; &C \delta_S^{p-1} \int_{S'} \int_{-\delta_S}^{\delta_S} \Big( |(\nabla v)\circ\bm\Psi|^p + |(\nabla^2 v)\circ\bm\Psi|^p \Big) \, dt \, d\bm z' \notag \\
	\le \; &C \delta_S^{p-1} ( \|\nabla v\|_{L^p(\bm\pi(S, \delta_S))}^p + \|\nabla^2 v\|_{L^p(\bm\pi(S, \delta_S))}^p ). \label{eq3: proof of nabla (v - v circ pi)}
\end{align}
where we have used H\"older's inequality and \eref{eq: local tubular neighborhood equivalence} in the third and fourth lines, respectively.
Substituting \eref{eq2: proof of nabla (v - v circ pi)} and \eref{eq3: proof of nabla (v - v circ pi)} into \eref{eq1: proof of nabla (v - v circ pi)}, we deduce that
\begin{equation*}
	\|\nabla_{\Gamma_h} (v - v\circ\bm\pi) \|_{L^p(S)} \le Ch_S^k \|\nabla v\|_{L^p(S)} + C \delta_S^{1-1/p} ( \|\nabla v\|_{L^p(\bm\pi(S, \delta_S))} + \|\nabla^2 v\|_{L^p(\bm\pi(S, \delta_S))} ).
\end{equation*}
Since $\|\nabla v\|_{L^p(\bm\pi(S, \delta_S))} \le C \delta_S^{1/p} \|\nabla v\|_{L^p(S)} + C \delta_S \|\nabla^2 v\|_{L^p(\bm\pi(S, \delta_S))}$ by \eref{eq2': boundary-skin estimates} and $h_S \le 1$, we conclude \eref{eq1: conclusion of lemma which bounds nabla(v - v circ pi)}.
Estimate \eref{eq2: conclusion of lemma which bounds nabla(v - v circ pi)} follows from \eref{eq1: conclusion of lemma which bounds nabla(v - v circ pi)} because $\|\nabla v\|_{L^p(S)} \le C(\|\nabla v\|_{L^p(\bm\pi(S))} + \delta_S^{1-1/p} \|\nabla^2 v\|_{L^p(\bm\pi(S, \delta_S))})$ by \eref{eq3: boundary-skin estimates} and \eref{eq1: equivalence of surface integrals}.
This completes the proof of \lref{lem: nablaGammah(v - vpi)}.

\section{Stability of $\mathcal I_h$} \label{sec: stability of Ih}
\begin{proof}[Proof of \lref{lem1: stability of Ih}]
	We focus only on the estimate on $S$; the one on $T$ can be proved similarly.
	It follows from \eref{eq: nodal basis estimate on S} and \eref{eq: Linfty norm of psip} that
	\begin{equation} \label{eq: Ihu in the Hm norm on S}
	\begin{aligned}
		\|\nabla_S^m (\mathcal I_h u)\|_{L^2(S)} &\le \sum_{\bm p \in \mathcal N_h \cap S} |(v, \psi_{\bm p})_{L^2(\sigma_{\bm p})}| \, \|\nabla_S^m \phi_{\bm p}\|_{L^2(S)}
			\le Ch_S^{-m + (d-1)/2} \sum_{\bm p \in \mathcal N_h \cap S} \|v\|_{L^1(\sigma_{\bm p})} \|\psi_{\bm p}\|_{L^\infty(\sigma_{\bm p})} \\
		&\le Ch_S^{-m - (d-1)/2} \sum_{S_1 \in \mathcal S_h(S)} \|v\|_{L^1(S_1)}
		\le Ch_S^{-m} \sum_{S_1 \in \mathcal S_h(S)} \|v\|_{L^2(S_1)},
	\end{aligned}
	\end{equation}
	where $\operatorname{meas}_{d-1}(S_1) \le Ch_S^{d-1}$ is used in the last line.
	By \pref{prop: transformation between S and Shat}(ii), for $S_1 \in \mathcal S_h(S)$ we have
	\begin{equation} \label{eq: v in the L2 norm on S1}
		\|v\|_{L^2(S_1)} \le Ch_S^{(d-1)/2} \|v \circ \bm F_S\|_{L^2(\hat S')} \le Ch_S^{(d-1)/2} \|v \circ \bm F_S\|_{H^1(\hat S')}
			\le C \sum_{l=0}^1 h_S^l \|\nabla_{S_1}^l v\|_{L^2(S_1)}
	\end{equation}
	where $\hat S'$ is a projected image of $\hat S := \bm F_{T_{S_1}}^{-1}(S_1)$.
	Substitution of \eref{eq: v in the L2 norm on S1} into \eref{eq: Ihu in the Hm norm on S} concludes \eref{eq: local stability of Ihv on S}.
\end{proof}

\begin{proof}[Proof of \lref{lem2: stability of Ih}]
	We focus on the estimate on $S$; the one on $T$ can be proved similarly.

	We introduce a reference macro element $\hat M_S' \subset \mathbb R^{d-1}$ which is a union of $\# \mathcal S_h(S)$ simplices of dimension $d-1$.
	By the regularity of the meshes, there is only a finite number of possibilities for $\hat M_S'$, which is independent of $h$ and $S$.
	There is a homeomorphism $\bm F_{M_S}: \hat M_S' \to M_S$ such that its restriction to each $(d-1)$-simplex $\hat S_1'$ belongs to $\mathbb P_k(\hat S_1')$ and is a $C^k$-diffeomorphism.
	
	For arbitrary $\hat P \in \mathbb P_k(\hat M_S')$, observe that $P := \hat P \circ \bm F_{M_S}^{-1} \in \operatorname{span}\{\phi_{\bm p}\}_{\bm p \in \mathcal N_h \cap M_S}$, and hence $(\mathcal I_h P)|_S = P|_S$.
	Therefore, it follows from \lref{lem1: stability of Ih} and \pref{prop: transformation between S and Shat}(ii) that
	\begin{align}
		\|\nabla_S^m (v - \mathcal I_h v)\|_{L^2(S)} &\le \|\nabla_S^m (v - P)\|_{L^2(S)} + \|\nabla_S^m \mathcal I_h(P - v)\|_{L^2(S)} \notag \\
			&\le C \sum_{l=0}^1 h_S^{l-m} \sum_{S_1 \in \mathcal S_h(S)} \|\nabla_{S_1}^l (v - P)\|_{L^2(S_1)} \notag \\
			&\le C \sum_{l=0}^1 h_S^{-m+(d-1)/2} \|\nabla_{\hat{\bm x}'}^l (v \circ \bm F_{M_S} - \hat P)\|_{L^2(\hat M_S')}. \label{eq1: proof of local interpolation error}
	\end{align}
	The Bramble--Hilbert theorem, combined with an appropriate choice of a constant function for $\hat P$, yields
	\begin{equation} \label{eq2: proof of local interpolation error}
		\|\nabla_{\hat{\bm x}'}^l (v \circ \bm F_{M_S} - \hat P)\|_{L^2(\hat M_S')} \le C \|\nabla_{\hat{\bm x}'} (v \circ \bm F_{M_S})\|_{L^2(\hat M_S')} 
			\le C h_S^{1-(d-1)/2} \|v\|_{H^1(M_S)}
	\end{equation}
	for $l = 0, 1$, where we have used \pref{prop: transformation between S and Shat}(ii) again.
	Now \eref{eq: local interpolation error on S} results from \eref{eq1: proof of local interpolation error} and \eref{eq2: proof of local interpolation error}.
\end{proof}
\begin{rem}
	The argument above closely follows that of \cite[p.\ 1899]{BerGir1998}.
	Because $\bm F_{M_S} \notin H^2(\hat M_S')$ in general, we considered the situation in which only $H^1(S_1)$-norms of $v$ appear.
\end{rem}

\section{Error estimates in the exact domain} \label{sec: estimate in exact domain}
\subsection{Extension of finite element functions}
We define a natural extension of $v_h \in V_h$ to $\Gamma(\delta)$ by
\begin{equation*}
	\bar v_h|_{\bm\pi(S, \delta) \setminus \Omega_h} = \hat v_h \circ \bm F_{T_S}^{-1} \quad (S \in \mathcal S_h).
\end{equation*}
Here note that $\hat v_h = v_h|_{T_S} \circ \bm F_{T_S} \in \mathbb P_k(\hat T)$ as well as $\bm F_{T_S} \in [\mathbb P_k(\hat T)]^d$ can be naturally extended (or extrapolated) to $2\hat T$, which doubles $\hat T$ by similarity with respect to its barycenter.
Because $h$ is sufficiently small, we may assume that $\bm F_{T_S}$ is a $C^k$-diffeomorphism defined in $2\hat T$ and that $\bm\pi(S, \delta) \subset \bm F_{T_S}(2\hat T) =: 2T_S$.
The extended function $\hat v_h \circ \bm F_{T_S}^{-1} \in C^\infty(2T_S)$ is denoted by $\overline{v_{h, T_S}}$.

\begin{rem} \label{rem: abuse of notation}
	The discrete extension satisfies only $\bar v_h \in L^2(\Omega \cup \Gamma(\delta))$ and $\bar v_h|_{\Gamma} \in L^2(\Gamma)$ globally, since $\bar v_h$ may be discontinuous across $\{ \bm x \in \Omega \setminus \Omega_h \mid \bm\pi(\bm x) \in \bm\pi(\partial S) \}$ for $S \in \mathcal S_h$ (i.e.\ the ``lateral part'' of the boundary of $\bm\pi(S, \delta) \setminus T_S$; cf.\ Figure \ref{fig1}).
	Nevertheless, for simplicity in reading, we will allow for the following abuse of notation:
	\begin{align*}
		\|\nabla(u - \bar u_h)\|_{L^2(\Omega)} &= \Big( \|\nabla(u - u_h)\|_{L^2(\Omega \cap \Omega_h)}^2 + \sum_{S \in \mathcal S_h} \|\nabla(u - \bar u_h)\|_{L^2(\bm\pi(S, \delta) \cap (\Omega \setminus \Omega_h))}^2 \Big)^{1/2}, \\
		\|\nabla(u - \bar u_h)\|_{L^2(\Gamma)} &= \Big( \sum_{S \in \mathcal S_h} \|\nabla_\Gamma (u - \bar u_h)\|_{L^2(\bm\pi(S))}^2 \Big)^{1/2}.
	\end{align*}
	Note, however, that $\bar v_h$ is continuous if $d = 2$ and the nodes of $\Gamma_h$ lie exactly on $\Gamma$.
\end{rem}

The discrete extension $\bar v_h$ can be estimated in $\Omega \setminus \Omega_h$ as follows.
\begin{lem} \label{lem: estimate in Omega minus Omegah}
	Let $v_h \in V_h$ and $S \in \mathcal S_h$. Then, for $p \in [1, \infty]$ we have
	\begin{align*}
		\|\overline{v_{h, T_S}}\|_{L^p(\bm\pi(S, \delta_S))} &\le C \Big( \frac{\delta_S}{h_S} \Big)^{1/p} \|v_h\|_{L^p(T_S)}, \\
		\|\nabla^m \overline{v_{h, T_S}}\|_{L^p(\bm\pi(S, \delta_S))} &\le C \Big( \frac{\delta_S}{h_S} \Big)^{1/p} h_S^{1-m} \|\nabla v_h\|_{L^p(T_S)} \quad (m \ge 1).
	\end{align*}
\end{lem}
\begin{proof}
	We focus on the case $m \ge 1$; the case $m = 0$ can be treated similarly.
	For simplicity, we write $T := T_S$.
	By the H\"older inequality,
	\begin{equation} \label{eq1: estimate in Omega minus Omegah}
		\|\nabla^m \overline{v_{h, T}}\|_{L^p(\bm\pi(S, \delta_S)} \le \operatorname{meas}_d(\bm\pi(S, \delta_S))^{1/p} \|\nabla^m \overline{v_{h, T}}\|_{L^\infty(2T)}
			\le C (h_S^{d-1} \delta_S)^{1/p} \|\nabla^m \overline{v_{h, T}}\|_{L^\infty(2T)}.
	\end{equation}
	Transforming to the reference coordinate, we have
	\begin{align}
		\|\nabla_{\bm x}^m \overline{v_{h, T}}\|_{L^\infty(2T)} &\le C \sum_{l=1}^m h_S^{-l} \|\nabla_{\hat{\bm x}}^l \hat v_h\|_{L^\infty(2\hat T)}
			\le C h_S^{-m} \sum_{l=1}^m \|\nabla_{\hat{\bm x}}^l \hat v_h\|_{L^\infty(\hat T)}
			\le C h_S^{-m} \|\nabla_{\hat{\bm x}} \hat v_h\|_{L^p(\hat T)} \notag \\
			&\le C h_S^{1-m-d/p} \|\nabla_{\bm x} v_h\|_{L^p(T)}, \label{eq2: estimate in Omega minus Omegah}
	\end{align}
	where we note that $\|\cdot\|_{L^\infty(\hat T)}$ and $\|\cdot\|_{L^\infty(2\hat T)}$ define norms for polynomials and that all norms of a finite dimensional space are equivalent to each other.
	Now substitution of \eref{eq2: estimate in Omega minus Omegah} into \eref{eq1: estimate in Omega minus Omegah} proves the desired estimate.
\end{proof}

For $S \in \mathcal S_h$, we will make use of the following trace inequality:
\begin{equation} \label{eq: local trace estimate}
	\|v\|_{L^2(S)} \le Ch_S^{-1/2} \|v\|_{L^2(T_S)} + C \|v\|_{L^2(T_S)}^{1/2} \|\nabla v\|_{L^2(T_S)}^{1/2} \quad \forall v \in H^1(T_S),
\end{equation}
as well as the inverse inequality
\begin{equation*}
	\|\nabla^m v_h\|_{L^2(S)} \le Ch_S^{m-1/2} \|v_h\|_{H^m(T_S)} \quad \forall v_h \in V_h, \; m \ge 0.
\end{equation*}

\subsection{$H^1$-error estimate in $\Omega$}
Let us prove that
\begin{equation*}
	\Big( \|\nabla(u - u_h)\|_{L^2(\Omega \cap \Omega_h)}^2 + \sum_{S \in \mathcal S_h} \|\nabla(u - \bar u_h)\|_{L^2(\bm\pi(S, \delta) \cap (\Omega \setminus \Omega_h))}^2 \Big)^{1/2} \le Ch^k \|u\|_{H^{k+1}(\Omega)}.
\end{equation*}
For this, by virtue of the $H^1(\Omega_h)$-error estimate in \tref{thm: H1 error estimate}, it suffices to show the following:

\begin{prop} \label{prop: H1 error estimate in Omega}
	Under the same assumptions as in \tref{thm: H1 error estimate} we have
	\begin{equation*}
		\Big( \sum_{S \in \mathcal S_h} \|\nabla(\tilde u - \bar u_h)\|_{L^2(\bm\pi(S, \delta_S) \setminus \Omega_h)}^2 \Big)
			\le Ch^{k+1}(\|\tilde u\|_{H^2(\Gamma(\delta))} + \|\tilde u\|_{H^{k+1}(\Omega_h)}) + Ch^{k/2} \|\nabla(\tilde u - u_h)\|_{L^2(\Omega_h)}.
	\end{equation*}
\end{prop}
\begin{proof}
	Fixing an arbitrary $S \in \mathcal S_h$ and setting $v_h = \mathcal I_h^L \tilde u$, we start from
	\begin{align*}
		\|\nabla(\tilde u - \bar u_h)\|_{L^2(\bm\pi(S, \delta_S) \setminus \Omega_h)} &\le \|\nabla(\tilde u - \overline{u_{h, T_S}})\|_{L^2(\bm\pi(S, \delta_S))} \\
			&\le \|\nabla(\tilde u - \overline{v_{h, T_S}})\|_{L^2(\bm\pi(S, \delta_S))} + \|\nabla(\overline{v_{h, T_S}} - \overline{u_{h, T_S}})\|_{L^2(\bm\pi(S, \delta_S))}.
	\end{align*}
	We apply \eref{eq2: boundary-skin estimates} to get
	\begin{equation*}
		\|\nabla(\tilde u - \overline{v_{h, T_S}})\|_{L^2(\bm\pi(S, \delta_S))} \le
			C \delta_S^{1/2} \|\nabla( \tilde u - v_h )\|_{L^2(S)} + C \delta_S \|\nabla^2( \tilde u - \overline{v_{h, T_S}} )\|_{L^2(\bm\pi(S, \delta_S))}.
	\end{equation*}
	Because of \eref{eq: Lagrange interpolation error estimate on S} and \eref{eq: local trace estimate}, we get
	\begin{equation} \label{eq1: H1 error in Omega minus Omegah}
		\|\nabla(\tilde u - \bar v_h)\|_{L^2(S)} \le Ch_S^{k-1/2} \|\tilde u\|_{H^{k+1}(T_S)}.
	\end{equation}
	By \lref{lem: estimate in Omega minus Omegah},
	\begin{align}
		\|\nabla^2(\tilde u - \overline{v_{h, T_S}})\|_{L^2(\bm\pi(S, \delta_S))}
		&\le \|\nabla^2 \tilde u\|_{L^2(\bm\pi(S, \delta_S))} + \|\nabla^2 \overline{v_{h, T_S}}\|_{L^2(\bm\pi(S, \delta_S))} \notag \\
		&\le \|\nabla^2 \tilde u\|_{L^2(\bm\pi(S, \delta_S))} \notag
			+ C \Big( \frac{\delta_S}{h_S} \Big)^{1/2} h_S^{-1} \|\nabla v_h\|_{L^2(T_S)} \notag \\
		&\le \|\nabla^2 \tilde u\|_{L^2(\bm\pi(S, \delta_S))} + C \delta_S^{1/2} h_S^{-3/2} \|\tilde u\|_{H^{k+1}(T_S)}. \label{eq2: H1 error in Omega minus Omegah}
	\end{align}
	Finally, observe that
	\begin{align}
		\|\nabla(\overline{v_{h, T_S}} - \overline{u_{h, T_S}})\|_{L^2(\bm\pi(S, \delta_S))} &\le C \Big( \frac{\delta_S}{h_S} \Big)^{1/2} \|\nabla(v_h - u_h)\|_{L^2(T_S)} \notag \\
			&\le C \Big( \frac{\delta_S}{h_S} \Big)^{1/2} (h_S^k \|\tilde u\|_{H^{k+1}(T_S)} + \|\nabla(\tilde u - u_h)\|_{L^2(T_S)}). \label{eq3: H1 error in Omega minus Omegah}
	\end{align}
	Now we deduce from \eref{eq1: H1 error in Omega minus Omegah}--\eref{eq3: H1 error in Omega minus Omegah} that
	\begin{align*}
		\|\nabla(\tilde u - \bar u_h)\|_{L^2(\bm\pi(S, \delta_S) \setminus \Omega_h)} &\le C \Big( \frac{\delta_S}{h_S} \Big)^{1/2} (h_S^k \|\tilde u\|_{H^{k+1}(T_S)} + \|\nabla(\tilde u - u_h)\|_{L^2(T_S)}) \\
			&\qquad + C \delta_S \|\nabla^2 \tilde u\|_{L^2(\bm\pi(S, \delta_S))} + C \Big( \frac{\delta_S}{h_S} \Big)^{3/2} \|\tilde u\|_{H^{k+1}(T_S)} \\
		&\le Ch_S^{k+1}(\|\tilde u\|_{H^2(\bm\pi(S, \delta_S))} + \|\tilde u\|_{H^{k+1}(T_S)}) + Ch_S^{k/2} \|\nabla(\tilde u - u_h)\|_{L^2(T_S)}.
	\end{align*}
	Taking the square and the summation over $S \in \mathcal S_h$, we obtain the desired estimate.
\end{proof}

\subsection{$H^1$-error estimate on $\Gamma$}
To establish
\begin{equation} \label{eq: global H1 error estimate on Gamma}
	\Big( \sum_{S \in \mathcal S_h} \|\nabla_\Gamma (u - \bar u_h)\|_{L^2(\bm\pi(S))}^2 \Big)^{1/2} \le Ch^k \|u\|_{H^{k+1}(\Omega; \Gamma)},
\end{equation}
we show the following estimate on each boundary mesh $S$.
\begin{prop} \label{prop: H1 error estimate on Gamma}
	In addition to the assumptions of \tref{thm: H1 error estimate} we assume $k > d/2$ so that $u \in H^{k+1}(\Omega) \hookrightarrow C^1(\overline\Omega)$.
	Then, for $S \in \mathcal S_h$ we have
	\begin{align*}
		\|\nabla_\Gamma (u - \bar u_h)\|_{L^2(\bm\pi(S))} &\le
			C \|\nabla_{\Gamma_h} (\tilde u - u_h) \|_{L^2(S)} + C \delta_S^{3/2} \|\tilde u\|_{H^3(\bm\pi(S, \delta_S))} \\
		&\qquad + C h_S^{(k-1)/2} \sum_{T_1 \in \mathcal T_h(T_S)} (h_S^k \|\tilde u\|_{H^{k+1}(T_1)} + \|\nabla (\tilde u - u_h)\|_{L^2(T_1)}),
	\end{align*}
	where $\mathcal T_h(T)$ means $\{T_1 \in \mathcal T \mid T_1 \cap T \neq \emptyset\}$ for $T \in \mathcal T_h$.
\end{prop}
\begin{proof}
	Since $h_S$ is sufficiently small, we may assume $\bm\pi(S) \subset  \Omega_h^c \cup \bigcup \mathcal T_h(T)$ with $T : = T_S$.
	Therefore,
	\begin{align*}
		\|\nabla_\Gamma (u - \bar u_h)\|_{L^2(\bm\pi(S))} = \Big( \|\nabla_\Gamma( u - \overline{u_{h, T}} )\|_{L^2(\bm\pi(S) \cap (T \cup \Omega_h^c))}^2 +
			\sum_{T_1 \in \mathcal T_h(T) \setminus \{T\}} \|\nabla_\Gamma(u - u_h)\|_{L^2(\bm\pi(S) \cap T_1)}^2 \Big)^{1/2}.
	\end{align*}
	Setting $v_h = \mathcal I_h^L \tilde u$, we have
	\begin{align*}
		\|\nabla_\Gamma( u - \overline{u_{h, T}} )\|_{L^2(\bm\pi(S))} &\le C\|\nabla_{\Gamma_h} [(u - \overline{u_{h, T}}) \circ \bm\pi] \|_{L^2(S)} \\
		&\le C(\|\nabla_{\Gamma_h} (\tilde u - u_h) \|_{L^2(S)} + \|\nabla_{\Gamma_h} [(\tilde u - v_h) - (u - \overline{v_{h, T}}) \circ \bm\pi] \|_{L^2(S)} \\
		&\qquad + \|\nabla_{\Gamma_h} [(v_h - u_h) - (\overline{v_{h, T}} - \overline{u_{h, T}}) \circ \bm\pi] \|_{L^2(S)}).
	\end{align*}
	The second and third terms in the right-hand side are majorized by
	\begin{align*}
		&C h_S^k \|\nabla (\tilde u - v_h)\|_{L^2(S)} + C \delta_S^{1/2} \|\nabla^2 (\tilde u - \overline{v_{h, T}})\|_{ L^2(\bm\pi(S, \delta_S)) } \\
		\le \; & Ch_S^{2k-1/2} \|\tilde u\|_{H^{k+1}(T)} + C \delta_S \|\nabla^2 (\tilde u - v_h)\|_{L^2(S)}
			+ C \delta_S^{3/2} \|\nabla^3 (\tilde u - \overline{v_{h, T}})\|_{ L^2(\bm\pi(S, \delta_S)) } \\
		\le \; & C h_S^{2k-1/2} \|\tilde u\|_{H^{k+1}(T)} + C \delta_S^{3/2} \|\tilde u\|_{ H^3(\bm\pi(S, \delta_S)) } +
			C \delta_S^{3/2} \Big( \frac{\delta_S}{h_S} \Big)^{1/2} h_S^{-2} \|\nabla v_h\|_{L^2(T)} \\
		\le \; & C h_S^{2k-1/2} \|\tilde u\|_{H^{k+1}(T)} + C \delta_S^{3/2} \|\tilde u\|_{H^3(\bm\pi(S, \delta_S))},
	\end{align*}
	and by
	\begin{align*}
		&C h_S^k \|\nabla (v_h - u_h)\|_{L^2(S)} + C \delta_S^{1/2} \|\nabla^2 (\overline{v_{h, T}} - \overline{u_{h, T}})\|_{L^2(\bm\pi(S, \delta_S))} \\
			\le \; &Ch_S^{k-1/2} \|\nabla (v_h - u_h)\|_{L^2(T)} + C \delta_S^{1/2} \Big( \frac{\delta_S}{h_S} \Big)^{1/2} h_S^{-1} \|\nabla (v_h - u_h)\|_{L^2(T)} \\
			\le \; &Ch_S^{k-1/2} (h_S^k \|\tilde u\|_{H^{k+1}(T)} + \|\nabla (\tilde u - u_h)\|_{L^2(T)}),
	\end{align*}
	respectively. Therefore, we obtain
	\begin{align*}
		\|\nabla_\Gamma( u - \overline{u_{h, T}} )\|_{L^2(\bm\pi(S) \cap (T \cup \Omega_h^c))} &\le \|\nabla_{\Gamma_h} (\tilde u - u_h) \|_{L^2(S)} +
			C h_S^{2k-1/2} \|\tilde u\|_{H^{k+1}(T_S)} + C \delta_S^{3/2} \|\tilde u\|_{H^3(\bm\pi(S, \delta_S))} \\
		&\qquad + C h_S^{k-1/2} \|\nabla (\tilde u - u_h)\|_{L^2(T_S)}.
	\end{align*}
	
	Next, for $T_1 \in \mathcal T_h(T)$ such that $T_1 \neq T$, we have
	\begin{equation*}
		\|\nabla_\Gamma(u - u_h)\|_{L^2(\bm\pi(S) \cap T_1)} \le \|\nabla_\Gamma(u - v_h)\|_{L^2(\bm\pi(S) \cap T_1)} + \|\nabla_\Gamma(v_h - u_h)\|_{L^2(\bm\pi(S) \cap T_1)}.
	\end{equation*}
	We estimate the first and second terms by
	\begin{align*}
		\operatorname{meas}_{d-1}(\bm\pi(S) \cap T_1)^{1/2} \|\nabla(\tilde u - v_h)\|_{L^\infty(T_1)}
			\le C (h_S^{d-2} \delta_S)^{1/2} \cdot C h_{T_1}^{k-d/2} \|\tilde u\|_{H^{k+1}(T_1)}
			\le C h_S^{(3k-1)/2} \|\tilde u\|_{H^{k+1}(T_1)}
	\end{align*}
	(note that $h_{T_1} \le C h_S$ by the regularity of the meshes), and by
	\begin{align*}
		\operatorname{meas}_{d-1}(\bm\pi(S) \cap T_1)^{1/2} \|\nabla(v_h - u_h)\|_{L^\infty(T_1)} &\le C (h_S^{d-2} \delta_S)^{1/2} \cdot C h_{T_1}^{-d/2} \|\nabla(v_h - u_h)\|_{L^2(T_1)} \\
			&\le C h_S^{(k-1)/2} (h_{T_1}^k \|\tilde u\|_{H^{k+1}(T_1)} + \|\nabla(\tilde u - u_h)\|_{L^2(T_1)}).
	\end{align*}
	Combining the estimates above all together, we conclude the desired inequality.
\end{proof}
\begin{rem}	
	(i) Because $\#\mathcal T_h(T) \le C$ for all $T \in \mathcal T_h$ by the regularity of the meshes, we can deduce \eref{eq: global H1 error estimate on Gamma} from the proposition above and \tref{thm: H1 error estimate}.
	
	(ii) The assumption $k > d/2$ excludes the case $k = 1$ even for $d = 2, 3$.
	However, a similar argument using $\nabla^2 v_h \equiv 0$ for $k = 1$ on each element yields
	\begin{align*}
		\|\nabla_\Gamma (u - \bar u_h)\|_{L^2(\bm\pi(S))} &\le
			C \|\nabla_{\Gamma_h} (\tilde u - u_h) \|_{L^2(S)} + C \delta_S^{1/2} \|\tilde u\|_{H^2(\bm\pi(S, \delta_S))} \\
		&\qquad + C \sum_{T_1 \in \mathcal T_h(T_S)} (h_S |T_1|^{d/2} \|\tilde u\|_{W^{2,\infty}(T_1)} + \|\nabla (\tilde u - u_h)\|_{L^2(T_1)}),
	\end{align*}
	under the additional regularity assumption $u \in W^{2,\infty}(\Omega)$.
\end{rem}

\subsection{Final result}
The final result including $L^2$-estimates is stated as follows.
\begin{thm}
	Let $k > d/2$.
	Assume that $f \in H^1(\Omega)$, $\tau \in H^{3/2}(\Gamma)$ if $k = 2$ and that $f \in H^{k-1}(\Omega)$, $\tau \in H^{k-1}(\Gamma)$ if $k \ge 3$.
	If $u$ and $u_h$ are the solutions of \eref{eq: continuous problem} and \eref{eq: FE scheme} respectively, then we have
	\begin{align*}
		\|\nabla (u - \bar u_h)\|_{L^2(\Omega)} + \|\nabla_\Gamma (u - \bar u_h)\|_{L^2(\Gamma)} &\le C_{f, \tau} h^k, \\
		\|u - \bar u_h\|_{L^2(\Omega)} + \|u - \bar u_h\|_{L^2(\Gamma)} &\le C_{f, \tau} h^{k+1},
	\end{align*}
	where $C_{f, \tau} := C (\|f\|_{H^{k-1}(\Omega)} + \|\tau\|_{H^{\max\{k-1, 3/2\}}(\Gamma)})$.
\end{thm}
\begin{proof}
	The $H^1$-estimates are already shown by Propositions \ref{prop: H1 error estimate in Omega} and \ref{prop: H1 error estimate on Gamma}.
	The $L^2(\Omega)$-estimate follows from
	\begin{equation*}
		\|\tilde u - \bar u_h\|_{L^2(\Gamma(\delta))} \le C \delta^{1/2} \|\tilde u - u_h\|_{L^2(\Gamma_h)} + C \delta \|\nabla (\tilde u - \bar u_h)\|_{L^2(\Gamma(\delta))}
			\le C_{f, \tau} (\delta^{1/2} h^{k+1} + \delta h^k),
	\end{equation*}
	where we have used the result of \tref{thm: L2 error estimate}.
	For the $L^2(\Gamma)$-estimate, we see that
	\begin{align*}
		\|u - \bar u_h\|_{L^2(\Gamma)} &\le C \|(u - \bar u_h) \circ \bm\pi\|_{L^2(\Gamma_h)} \le
			C \|\tilde u - u_h\|_{L^2(\Gamma_h)} + C \|(\tilde u - u_h) - (u - \bar u_h) \circ \bm\pi\|_{L^2(\Gamma_h)} \\
		&\le C_{f, \tau} h^{k+1} + C \delta^{1/2} \|\nabla (\tilde u - \bar u_h)\|_{L^2(\Gamma(\delta))} \\
		&\le C_{f, \tau} (h^{k+1} + \delta^{1/2} h^k),
	\end{align*}
	where we have used \pref{prop: H1 error estimate in Omega} together with \tref{thm: H1 error estimate} in the last line.
\end{proof}
\end{document}